\documentclass[onecolumn]{autart}               

\usepackage{amssymb,amsmath,color,psfrag}
\usepackage[draft]{graphicx}
\usepackage{algorithm,algorithmic,xspace}
\usepackage{epsfig}
\usepackage{mathrsfs}
\usepackage{amsfonts}
\usepackage{verbatim}
\usepackage{subfig}

\newcommand{\real}{{\mathbb{R}}}
\newcommand{\subscr}[2]{{#1}_{\textup{#2}}}

\newcommand{\map}[3]{#1: #2 \rightarrow #3}

\newcommand{\norm}[2]{\|#1\|_{#2}}

\newcommand{\until}[1]{\{1,\dots,#1\}}

\newcommand{\subj}{\text{subj. to}}

\newcommand{\minimize}{\text{minimize}}

\newcommand{\dom}[1]{\text{dom~#1}}

\newcommand{\cball}[2]{\overline{B}(#2,#1)}

\renewcommand{\cball}[2]{B(#2,#1)}

\renewcommand{\AA}{\mathcal{A}}

\newcommand{\CC}{\mathcal{C}}

\newcommand{\FF}{\mathcal{F}}
\newcommand{\GG}{\mathcal{G}}

\newcommand{\NN}{\mathcal{N}}
\newcommand{\PP}{\mathcal{P}}

\renewcommand{\SS}{\mathcal{S}}
\newcommand{\TT}{\mathcal{T}}
\newcommand{\UU}{\mathcal{U}}

\newcommand{\PVTOLo}{\text{PVTOL}_{0}}

\newcommand{\EpsVTOL}{\epsilon_{\mbox{\tiny\sc{PVTOL} } } }
\renewcommand{\EpsVTOL}{\epsilon_{\mbox{\tiny\sc{P} } } }
\newcommand{\EpsVTOLo}{\epsilon_{\mbox{\tiny\sc{P0} } } }

\newcommand{\avec}{{\bf{a}}}

\newcommand{\rr}{r}

\newcommand{\XU}{\mathcal{XU}}
\newcommand{\cXU}{\overline{\mathcal{XU}}}

\newcommand{\rhoFR}{\rho}

\newcommand{\EpsBeta}{\epsilon_{\mbox{\tiny c} } }
\newcommand{\DelBeta}{\delta_{\mbox{\tiny c} } }

\newcommand{\lift}{\texttt{LIFT}}
\newcommand{\ponewt}{\texttt{PO\_Newt}}

\newcommand{\Xinf}{X_\infty}
\newcommand{\Xtilde}{\widetilde{X}}
\newcommand{\yout}{{\bf y}}

\newtheorem{theorem}{Theorem}[section]
\newtheorem{proposition}[theorem]{Proposition}

\newtheorem{lemma}[theorem]{Lemma}

\newtheorem{assumption}[theorem]{Assumption}

\theoremstyle{definition}
\newtheorem{definition}[theorem]{Definition}
\newtheorem{remark}[theorem]{Remark}

\newenvironment{proof}{\textit{Proof: }}{\hfill \QED \\ \par}

\def\QEDclosed{\mbox{\rule[0pt]{1.3ex}{1.3ex}}} 

\def\QED{\QEDclosed} 

\newcommand\oprocendsymbol{\hbox{$\square$}}
\newcommand\oprocend{\relax\ifmmode\else\unskip\hfill\fi\oprocendsymbol}

\renewcommand{\theenumi}{(\roman{enumi})}

\graphicspath{{figs/}}

\begin{document}

\begin{frontmatter}


  \title{Computing feasible trajectories for constrained\\ maneuvering systems:
    the PVTOL example\thanksref{footnoteinfo}} 

  \thanks[footnoteinfo]{This paper was not presented at any IFAC
    meeting. Preliminary versions of this work with partial results were
    presented as \cite{GN-JH-RF:05,GN-JH-RF:07}.  Corresponding
    author G. Notarstefano, giuseppe.notarstefano@unile.it. Tel. +390832297360
    Fax +39 0832 297733.}


  \author[lecce]{Giuseppe Notarstefano}
  \and \author[boulder]{John Hauser}
  \address[lecce]{Departement of Engineering, University of Lecce, Via per
    Monteroni, 73100 Lecce, Italy {\tt\small giuseppe.notarstefano@unile.it}}
  \address[boulder]{Department of Electrical and Computer Engineering,
    University of Colorado, Boulder, CO 80309-0425, USA {\tt\small
      hauser@colorado.edu}}



\begin{abstract}          
  In this paper we provide an optimal control based strategy to explore
  \emph{feasible trajectories} of nonlinear systems, that is to find curves that
  satisfy the dynamics as well as point-wise state-input constraints.  The
  strategy is interesting itself in understanding the capabilities of the system
  in its operating region, and represents a preliminary tool to perform
  trajectory tracking in presence of constraints. The strategy relies on three
  main tools: dynamic embedding, constraints relaxation and novel optimization
  techniques, introduced in \cite{JH:02,JH-AS:06}, to find regularized solutions
  for \emph{point-wise} constrained optimal control problems.  The strategy is
  applied to the PVTOL, a simplified model of a real aircraft that captures the
  main features and challenges of several ``maneuvering systems''.
\end{abstract}

\end{frontmatter}

\section{Introduction}
%
In several fields as aerospace, robotics and automotive, designers have to deal
with complex nonlinear system dynamics. A deep knowledge of the behavior of such
systems is fundamental both in controlling them about a (possibly aggressive)
desired trajectory and in assisting the engineer in the design process.
An interesting problem is, therefore, the exploration of the trajectory manifold
of the system, that is, the characterization of the system (state-input)
trajectories and their parametrization with respect to `` performance-output
curves''. More formally, given a desired curve for some of the states (outputs),
we aim at finding a state-input \emph{lifted trajectory} (i.e. a state-input
curve satisfying the dynamics) whose outputs are close to the desired ones.
The solution of this problem is interesting itself in understanding the behavior
of the system and provides a nominal trajectory that can be used in a
receding horizon scheme for trajectory tracking.

Having in mind engineering applications, there is an important aspect to take
into account in the solution of exploration and tracking, that is the presence
of constraints in the system. Such constraints may arise from diverse causes
such as: physical bounds on the states and the inputs, validity bounds of the
model or presence (respectively absence) of important properties
(e.g. controllability). We will call the region where constraints are satisfied
\emph{feasibility region}.
This implies that, not only we look for lifted trajectories (curves satisfying the
dynamics), but furthermore we ask for \emph{feasible lifted trajectories}, that is
trajectories lying in the feasibility region.

In this paper we concentrate our attention on a special class of nonlinear
systems that we call \emph{maneuvering systems}. In this class we include those
systems for which a natural notion of performance outputs is present. Namely,
some of the states are requested to follow (almost exactly) a desired profile,
while the remaining states meet suitable (feasibility) bounds.
Maneuvering systems include for example vehicles (cars, motorcycles, aerial
vehicles, marine vehicles), manipulators and several other mechanical systems.
In this paper we consider as ``prototype example'' of maneuvering system the
PVTOL aircraft.
The PVTOL was introduced by Hauser et al. in \cite{JH-SSS-GM:92} in order to
capture the lateral non-minimum phase behavior of a Vertical Take Off and
Landing (VTOL) aircraft. This model has been widely studied in the literature
for its property of combining important features of nonlinear systems with
``tractable'' equations. Furthermore, the dynamics of many other mechanical
(maneuvering) systems can be rewritten in a similar fashion, e.g., the cart-pole
system, the pendubot \cite{MWS-DJB:95}, the bicycle model \cite{NHG-JEM:95},
\cite{JH-AS-RF:04} and the longitudinal dynamics of a real aircraft.
Since the PVTOL has been introduced in 1992, many researchers have studied this
system. 
A non exhaustive literature review of works on trajectory tracking or path
following of the PVTOL includes
\cite{JH-SSS-GM:92,PM-SD-BP:96,SAA-NHM:02,LM-AI-AS:02,LC-MT:07,LC-MM-CN-MT:10}.

The trajectory exploration problem has been investigated in the literature in
different formulations. Trajectory exploration for a class of VTOL aircrafts is
tackled in \cite{RN-LM:11}, where an optimization strategy is proposed to
compute optimal transition maneuvers. The problem of finding a (state-input)
trajectory whose output is exactly an assigned desired curve is known in the
literature as nonlinear inversion. This problem is particularly challenging for
nonminimum-phase systems. The problem was introduced and solved for some classes
of nonlinear systems and desired output curves in \cite{SD-DC-BP:96} and
extended to time-varying and non-hyperbolic systems respectively in
\cite{SD-BP:98} and \cite{SD:99}. More recently, \cite{AP-KYP:08}, a new
approach based on the notion of convergent systems was proposed to solve the
problem. In \cite{JH-AS-RF:05} the nonlinear inversion problem was solved for an
inverted pendulum by use of exponential dichotomy under mild conditions on the
output curve. Nonlinear inversion is strongly related to the output regulation
problem, that is to the design of a control law such that the system output
asymptotically tracks a desired output curve. An early reference is
\cite{AI-CIB:90}. There, the nonlinear inversion problem was solved by means of
a suitable partial differential equation when the desired output curve is the
trajectory of an exosystem. In \cite{LRH-GM:97} a two-step strategy was proposed
to solve the stable inversion problem.
An overview on the topic can be found in \cite{CIB-AI:00}, whereas more recent
references include \cite{AI-LM-AS:03,JH:04,AP-NVW-HN:07}.



%
%

The contribution of the paper is threefold. First, we propose an optimal control
based strategy to compute feasible trajectories of maneuvering systems. The
strategy relies on three main ideas: dynamic embedding, constraints relaxation
and continuation with respect to (the embedding and relaxation) parameters.
In detail, we compute feasible trajectories (i.e., trajectories satisfying
pointwise state and input constraints) that minimize a weighted $L_2$ distance
from the desired output curve.  In order to compute an approximate solution to
this problem we perform the following steps.
We embed the system into a family of systems and relax the feasibility region so
that the constraints are not active. For each value of the system (embedding)
and constraint (relaxation) parameters, we compute an unconstrained lifted
(state-input) trajectory (by applying an optimal control based dynamic embedding
technique) and use it as desired curve for a constrained $L_2$ distance
minimization.
To compute a feasible trajectory ($L_2$ close to the unconstrained one), we
design a relaxed version of the constrained optimal control problem. The
relaxation is based on the introduction of a parametrized barrier functional to
handle the constraints \cite{JH-AS:06}. The resulting optimal control problem is
solved by means of a projection operator based Newton method \cite{JH:02}.
%
%
The final ingredient of the strategy is a continuation procedure to update the
embedding and relaxation parameters up to their nominal values.

Second, we prove the effectiveness of the strategy, namely that a feasible
lifted trajectory can be computed, for suitable values of the embedding and
relaxation parameters.
The proof of this result relies on the continuity and differentiability of an
optimal control minimizer with respect to parameters, which is provided as a
stand alone result.
An analogous result was already proven in \cite{HM-HJP:94} for unconstrained
systems and extended to input-state constrained systems in \cite{HM-HJP:95}. The
main differences with the existing results are the following. We do not consider
input and state constraints directly, but take them into account in a relaxed
version of the constrained optimal control problem by use of a barrier
functional, the barrier functional being weighted by one of the varying
parameters. Furthermore, we take into account the system dynamics by means of a
trajectory tracking projection operator. The projection operator is the key
distinctive feature for the proof of the differentiability result. Indeed, the
projection operator allows to convert the dynamically constrained optimal
control problem into an unconstrained trajectory optimization problem. Thus, an
appropriate implicit function theorem can be used to \emph{solve} the first
order necessary condition equation. The implicit function theorem allows to show
that, if the second derivative of the cost composed with the projection operator
is invertible at the nominal parameter, then there is a neighborhood on which
the local minimizer exists and is $\CC^1$ with respect to the parameter. It
turns out that the appropriate condition to ensure invertibility of the operator
is that the minimizer satisfies the second order sufficiency condition at the
nominal parameter.

Third and final, we provide a complete characterization of the exploration
strategy for the PVTOL. In detail, we first solve the nonlinear inversion problem for
the unconstrained PVTOL. Given any $\CC^4$ desired output curve resulting into a
bounded acceleration profile, we prove that a trajectory can be computed for the
decoupled system and for suitable positive values of the coupling parameter. 
Based on this result and on the second set of contributions, we show that all
the strategy steps can be performed for suitable values of the feasibility
region and the coupling parameter. Finally, we perform a numerical
analysis showing that, in fact, feasible trajectories of the PVTOL can be
computed for aggressive desired output curves even in presence of relatively
tight constraints.

The paper is organized as follows. In Section~\ref{sec:PVTOL-model} the notion
of maneuvering systems is introduced and the PVTOL aircraft is presented as a
prototype example. Section~\ref{sec:devel-perf-tasks} defines the performance
tasks solved in the paper, namely unconstrained and constrained trajectory
lifting. In Section~\ref{sec:unconstr_lift} the unconstrained lifting task is
solved for the PVTOL aircraft. That is, a lifted trajectory is proven to exists
for some positive values of the coupling parameter. In
Section~\ref{sec:explor_strategy} the proposed optimal control based exploration
strategy is presented and in Section~\ref{sec:strategy-analysis} a theoretical
analysis of the strategy is developed, proving that for suitable values of the
system and constraint parameters a feasible trajectory exists. In
Section~\ref{sec:numerical_comp} numerical computations are provided showing the
effectiveness of the strategy for the constrained PVTOL on an aggressive barrel
roll trajectory in presence of respectively input and state-input
constraints. Finally, in Appendix~\ref{sec:prelim_proj_oper} an overview of the
trajectory tracking projection operator theory and the projection operator based
Newton method for unconstrained and constrained optimal control problems is
given.

\paragraph*{Notation}
For a function $\map{g}{[0,T]}{\real^p}$, $T>0$, we let
$||g(\cdot)||_{L_\infty}=\sup_{t\in[0,T]} ||g(t)||$ be the usual $L_\infty$
norm.
Let $\CC^k[0,T]^p$, $L_\infty[0,T]^p$ and $L_2[0,T]^p$ be the spaces of
functions $\map{g}{[0,T]}{\real^p}$ that are respectively $k$ times
differentiable with continuous $k$-th derivative, bounded and Lebesgue
integrable, and square integrable on $[0,T]$. In the rest of the paper we will
abuse notation and denote them as $\CC^k$, $L_\infty$ and $L_2$ when domain and
codomain are clear. Let $\xi \mapsto A(\xi)$ be a twice Fr\'echet differentiable
operator, we denote respectively $\zeta \mapsto D A(\xi_0) \cdot \zeta$ and
$(\zeta,\eta) \mapsto D^2 A(\xi_0) \cdot (\zeta,\eta)$ the first and second
Fr\'echet differentials of $A$ at $\xi_0$.
Given a control system $\dot x = f(x, u)$, where $x \in \real^n$ is the state
and $u \in \real^m$ is the control input, we say that a \emph{bounded} curve
$\eta = (\bar{x}(\cdot), \bar{u}(\cdot))$ is a (state-control) \emph{trajectory}
of the system if $\dot{\bar{x}}(t) = f(\bar{x}(t), \bar{u}(t))$ for all $t\in
[0,T]$, $0 < T \leq +\infty$, and $x(0) = \bar{x}(0)$.
%
%
Trajectories of the system through $x_0$ belong to the affine subspace $\Xtilde
:= (x_0, 0) + \Xinf$, where $\Xinf$ is the closed subspace of
$L^{n+m}_{\infty}[0, T]$ of curves $\zeta = (\beta, \nu)$ with continuous
$\beta$, $\beta(0) = 0$, and bounded $\nu$.
We denote $\TT\subset \Xtilde$ the set of bounded (in $L_\infty$) trajectories
through $x_0$.



\section{The PVTOL model and maneuvering system definition}
\label{sec:PVTOL-model}%
In this section we introduce the system that motivates our work, the PVTOL, and
inspired by this model we introduce the notion of maneuvering systems.
The PVTOL aircraft was introduced in \cite{JH-SSS-GM:92}. Using
standard aeronautic conventions the equations of motion are given by
\begin{equation}
  \begin{array}{cll}
    \ddot y & = & \phantom{-}u_1 \sin \varphi - \EpsVTOL u_2 \cos \varphi\\
    \ddot z & = &           -u_1 \cos \varphi - \EpsVTOL u_2 \sin \varphi + g\\
    \ddot \varphi & = & \phantom{-}u_2.
  \end{array}
  \label{eq:PVTOL}
\end{equation}
The aircraft state is given by the position $(y, z)$ of the center of gravity,
the roll angle $\varphi$ and the respective velocities $\dot y$, $\dot z$ and
$\dot \varphi$. The control inputs $u_1$ and $u_2$ are respectively the vertical
thrust force and the rolling moment. The gravity acceleration is denoted by $g$.
An interesting feature of the PVTOL model is that the rolling moment $u_2$
generates also a lateral force $\EpsVTOL u_2$, where $\EpsVTOL \in \real$ is a
coupling coefficient. 
%
%
In Figure~\ref{figure:PVTOL} the PVTOL aircraft with the reference system and
the inputs is shown.
\begin{figure}[thpb]
  \centering
  \includegraphics[width = 0.22\linewidth]{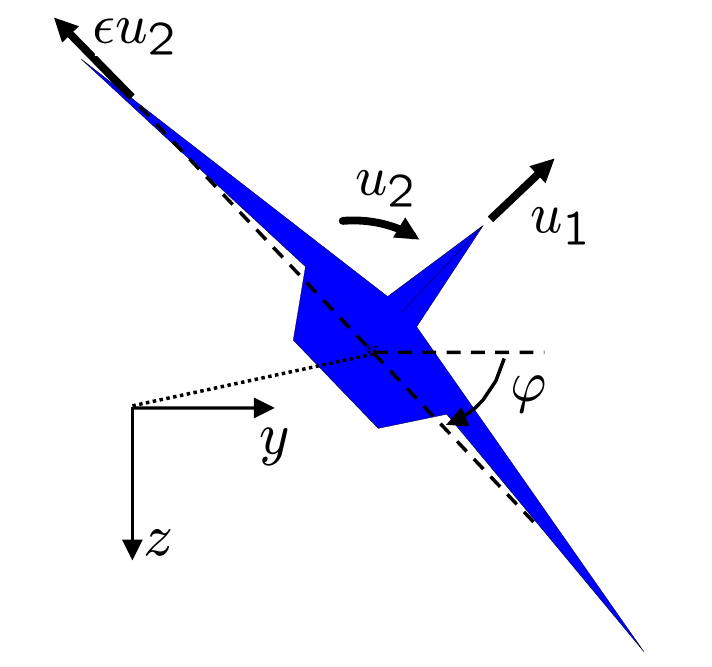}
  \caption{PVTOL aircraft.}
  \label{figure:PVTOL}
\end{figure}
Depending on the value of $\EpsVTOL$ the PVTOL shows very diverse dynamic
behaviors and possesses different control properties. We will clarify them in
the next sections.

Next, we exploit an important feature of the PVTOL that can be generalized to a
wide class of systems that we will refer to as \emph{maneuvering systems}. In
analyzing the PVTOL, we can partition the state space into the cartesian
product of: i) \emph{external position states} ($y$ and $z$), ii) \emph{internal
  position states} ($\varphi$) and iii) \emph{velocities} ($\dot{y}$, $\dot{z}$
and $\dot\varphi$).
In many engineering applications the objective is to track a time parametrized
curve described by the external position states while maintaining the internal
position states bounded (the trajectories of the velocities being consistent
with the related positions).
Thus, a natural choice of outputs arises for these systems, namely the
external position states. We call these outputs \emph{performance outputs}
meaning that a task for the system may be defined by assigning a desired curve
for these states (desired performance outputs)\footnote{The performance outputs
  are different from measured outputs, i.e., states or functions of states that
  can be measured by a sensor. Measured outputs play an important role in
  control design, but will not be considered here.}.
From now on, since we will
be only dealing with performance outputs, we will refer to them as
outputs.
A state space model of a maneuvering system is given by
\[
\begin{split}
  \dot{x}(t) &= f(x(t),u(t)),\\
  \yout(t) &= p(x(t)), \qquad t\in[0,T], \,T>0,
\end{split}
\]
where $x$, $u$ and $\yout$ are respectively the state, the input and the
performance output, $\map{f}{\real^n\times\real^m}{\real^n}$ is assumed to be
$\CC^2$ in both arguments, and $p$ selects a subset of the states
(e.g. the external position states for the PVTOL).

For the PVTOL, a state space model can be obtained by posing $x =
(y,z,\varphi,\dot{y},\dot{z},\dot{\varphi})$ and $u = (u_1,u_2)$. Also, a
natural choice of performance outputs is given by the position of the center of
gravity $(y,z)$, so that $p(x) := (y,z) = (x_1,x_2)$.

\section{Development of performance tasks}
\label{sec:devel-perf-tasks}
In this section we identify important challenges that arise in studying the
PVTOL capabilities and that can be generalized to maneuvering systems. These
challenges will drive us in providing useful strategies to explore the dynamic
capabilities of maneuvering systems.

We start defining the first two tasks we are interested in. Informally, given a
time-parametrized desired output curve, we want to find a (state-control)
trajectory of the system, such that the output portion of the trajectory is
close to the desired output curve. We will call this task \emph{trajectory
  lifting}. We will also define an approximate version of such task, called
\emph{practical trajectory lifting}, suitable for computations. More formally we
can define the two tasks as follows.
Given an output curve $\alpha(\cdot) \in L_\infty[0,T]^{p}$, we denote
$\norm{\alpha(\cdot)}{L_2}$ a suitable weighted $L_2$ norm of $\alpha(\cdot)$,
that is, $\norm{\alpha(\cdot)}{L_2} = \int_0^T \alpha(t)^T W \alpha(t) dt +
\alpha(T)^T W_1 \alpha(T)$, with $W$ and $W_1$ positive definite
matrices.\footnote{If $W= I_n$ and $W_1 = 0$ this is the classical $L_2$ norm.}

\begin{definition}[Trajectory lifting task]
  Let $\subscr{\yout}{d}(t)$, $t\in[0, T]$, be a desired sufficiently
  smooth output curve. Find a bounded trajectory $(x^*(\cdot),
  u^*(\cdot))\in\TT$ such that
  \[
  \norm{p(x^*(\cdot)) - \subscr{\yout}{d}(\cdot)}{L_2}^2 \leq \norm{p(x(\cdot))
    - \subscr{\yout}{d}(\cdot)}{L_2}^2 \qquad \text{for all } (x(\cdot),
  u(\cdot))\in\TT
  \]
\end{definition}

\begin{definition}[Practical trajectory lifting task]
  Let $\subscr{\yout}{d}(t)$, $t\in[0, T]$, be a desired sufficiently
  smooth output curve.  For a given $\epsilon>0$, find a trajectory
  $(x_\epsilon^*(\cdot), u_\epsilon^*(\cdot))\in\TT$ such that
  \[
  \norm{p(x_\epsilon^*(\cdot)) - \subscr{\yout}{d}(\cdot)}{L_2}^2 \leq
  \norm{p(x(\cdot)) - \subscr{\yout}{d}(\cdot)}{L_2}^2 + \epsilon \qquad
  \text{for all } (x(\cdot), u(\cdot))\in\TT
  \]
\end{definition}

\begin{remark}[Trajectory lifting and dynamic inversion]
  Dynamic inversion,\cite{SD-DC-BP:96}, is a trajectory lifting problem in the
  case the time horizon is $(-\infty, +\infty)$. The objective is to find a
  (state-input) trajectory such that the output trajectory is exactly the
  desired one. Clearly, if a solution to dynamic inversion exists, it is also a
  solution for the trajectory lifting task on the time horizon $[0, T]$ with
  zero minimum cost. If such a trajectory exists we say that it \emph{exactly}
  solves the trajectory lifting task. In the next section we will show that for
  the PVTOL it is in fact possible to solve the dynamic inversion problem. This
  could be not the case for other maneuvering systems, but (practical)
  trajectory lifting could still be solved by using optimal control.
\end{remark}

\begin{remark}[Trajectory lifting for flat systems]
  In most cases the performance outputs are driven by the application and cannot
  be decided by the designer. Thus, even for systems that are feedback
  linearizable or differentially flat, the design of a lifted trajectory is an
  issue.
  Regarding the PVTOL, in \cite{JH-SSS-GM:92} it was shown that the decoupled
  PVTOL ($\EpsVTOL=0$) was feedback linearizable, hence differentially flat,
  relative to the natural outputs. In \cite{PM-SD-BP:96} it was shown that the
  coupled PVTOL is also differentially flat, but with respect to the flat
  outputs $\subscr{y}{f} = y + \EpsVTOL \sin\varphi$ and $\subscr{z}{f} = z +
  \EpsVTOL \cos\varphi$.  If the flat outputs were chosen as performance
  outputs, the problem of finding a trajectory of the system consistent with the
  outputs would be easily solved as for the decoupled model. However, physical
  considerations suggest that the natural outputs are better suited as
  performance outputs.
\end{remark}


In this paper we are interested in a more challenging task, namely a constrained
version of the lifting task. That is, we want to perform the lifting task while
enforcing point-wise constraints on control inputs and states. In other words,
given a desired output curve and a region of the state-input space, we want to
find a trajectory that lies entirely in the region and whose output portion is
close (according to a given cost function) to the desired curve.
%

More formally, we define a \emph{feasibility region} $\cXU\subset
\real^n\times\real^m$ as a a compact simply connected region of the state-input
space where the trajectories of the system must lie at every time.
Consistently, a \emph{feasible trajectory} for $\cXU$ is a trajectory of the
system, $(x(\cdot), u(\cdot))\in\TT$, such that $(x(t), u(t))\in\cXU$ for almost
all $t\in[0,T]$. 
In the rest of the paper we will focus on trajectories that belong to the
interior, $\XU$, of $\cXU$. Thus, we say that a trajectory of the system,
$(x(\cdot), u(\cdot))\in\TT$, is a \emph{strictly feasible trajectory} for
$\cXU$ if $(x(t), u(t))\in\XU$ for almost all $t\in[0,T]$.
%
We are now ready to define the constrained version of the lifting task. As for
the unconstrained problem we define an exact and a practical task.

\begin{definition}[Feasible trajectory lifting task]
  \label{def:feas_traj_lift}
  Let $\subscr{\yout}{d}(t)$, $t\in[0, T]$, be a desired sufficiently
  smooth output curve and $\cXU\subset\real^n\times\real^m$ a feasibility
  region. Find a feasible trajectory, $(x^*(\cdot), u^*(\cdot))\in\TT$ with
  $(x_\epsilon^*(t), u_\epsilon^*(t)) \in\cXU$ for almost all $t\in [0,T]$, such
  that
  \[
  \norm{p(x^*(\cdot)) - \subscr{\yout}{d}(\cdot)}{L_2}^2 \leq \norm{p(x(\cdot))
    - \subscr{\yout}{d}(\cdot)}{L_2}^2
  \]
  for all $(x(\cdot), u(\cdot))\in\TT$ with $(x(t), u(t)) \in\cXU$ for almost
  all $t\in [0,T]$.
\end{definition}

\begin{definition}[Practical feasible trajectory lifting task]
  \label{def:pract_feas_traj_lift}
  Let $\subscr{\yout}{d}(t)$, $t\in[0, T]$, be a desired sufficiently
  smooth output curve and $\cXU\subset\real^n\times\real^m$ a feasibility
  region. For a given $\epsilon>0$, find a feasible trajectory,
  $(x_\epsilon^*(\cdot), u_\epsilon^*(\cdot))\in\TT$ with $(x_\epsilon^*(t),
  u_\epsilon^*(t)) \in\cXU$ for almost all $t\in [0,T]$, such that
  \begin{equation}
    \norm{p(x_\epsilon^*(\cdot)) - \subscr{\yout}{d}(\cdot)}{L_2}^2 \leq
    \norm{p(x(\cdot)) - \subscr{\yout}{d}(\cdot)}{L_2}^2 + \epsilon
    \label{eq:feas_task_pract}
  \end{equation}
  for all $(x(\cdot), u(\cdot))\in\TT$ with $(x(t), u(t)) \in\cXU$ for almost
  all $t\in [0,T]$.
\end{definition}

Finding a global solution to the above problems is a hard task since we are
dealing with infinite dimensional optimization problems. Thus, our goal in this
paper is to find a feasible trajectory that satisfies locally
equation~\eqref{eq:feas_task_pract}.

\section{Trajectory lifting for the unconstrained PVTOL}
\label{sec:unconstr_lift}
In this section we solve the exact lifting task for the coupled PVTOL with
positive values of the parameter $\EpsVTOL$ and show that the lifted trajectory
depends continuously on the parameter.

\subsection{Trajectory lifting for the decoupled PVTOL model}
The exact lifting task can be easily solved for the decoupled PVTOL model, that
is for the model with $\EpsVTOL = 0$. Since we will often refer to this special
case, we use for it the ad hoc notation $\PVTOLo$.
In \cite{JH-SSS-GM:92} the $\PVTOLo$ was shown to be input-output linearizable
provided $u_1 \neq 0$. Here, we provide sufficient conditions to compute a
trajectory of the system given a $\CC^4$ desired output curve.
We rewrite the equation for the $\PVTOLo$ since this will play an important role
in the development of our strategy.
\begin{equation}
  \begin{array}{cll}
    \ddot y & = & \phantom{-}u_1 \sin \varphi\\
    \ddot z & = &           -u_1 \cos \varphi + g\\
    \ddot \varphi & = & \phantom{-}u_2.
  \end{array}
  \label{eq:PVTOL0}
\end{equation}
The following assumption will be used in the paper.
\begin{assumption}[Annulus assumption]
  Let $\subscr{\yout}{d}(\cdot)\in\CC^4$ be a desired output curve. Let
  $\subscr{\avec}{d}(t) := (\subscr{\ddot{y}}{d}(t),
  \subscr{g-\ddot{z}}{d}(t))^T$, assume that $\subscr{\avec}{d}(t)\neq0$ and
  $0<\subscr{a}{min}\leq \norm{\subscr{\avec}{d}}\leq \subscr{a}{min}$ for all
  $t$.
\end{assumption}

A graphical interpretation of the annulus assumption is depicted in
Figure~\ref{fig:acc_range}. We ask the vector $\subscr{\avec}{d}$ to lie in the
annulus of radiuses $\subscr{a}{min}$ and $\subscr{a}{max}$ centered at the
origin of the reference axes $\overrightarrow{\ddot{y}}$ and
$\overrightarrow{\ddot{z}-g}$.

\begin{figure}[h]
  \centering
  \includegraphics[width = 0.3\textwidth]{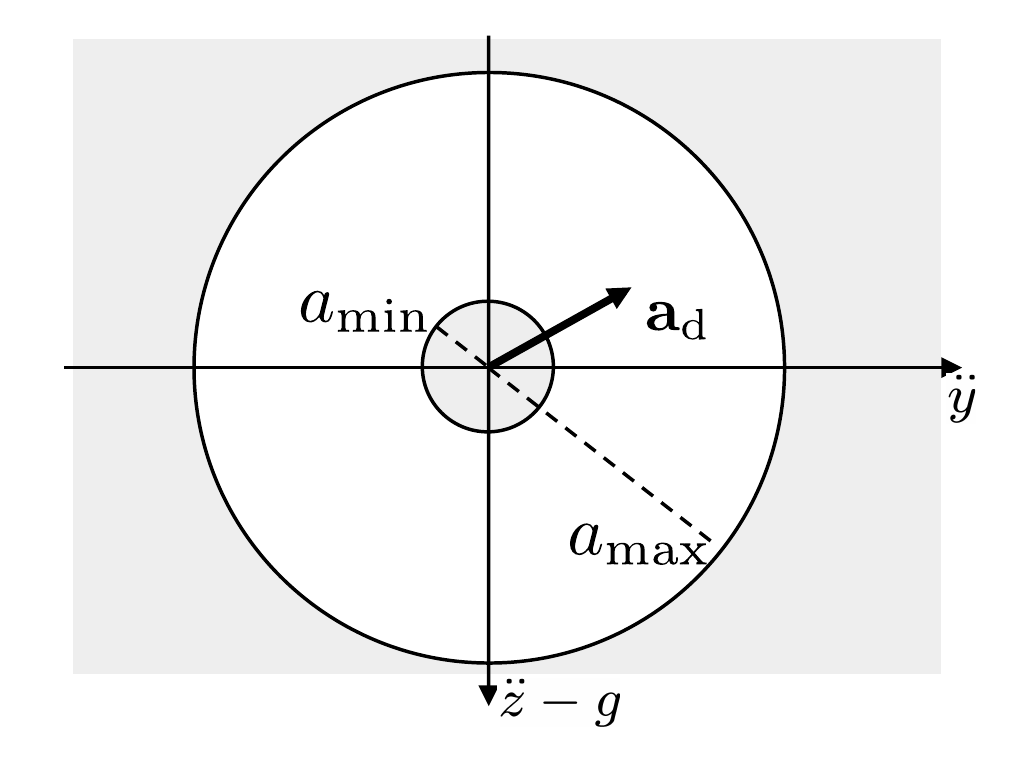}
  \caption{Acceleration range}
  \label{fig:acc_range}
\end{figure}

Under this assumption the (state-input) trajectory of the $\PVTOLo$ may be parametrized in
terms of the desired output curve by
\begin{equation}
  \begin{split}
    \varphi_0(t) &= \angle{\subscr{\avec}{d}(t)} = \text{atan2}\big(\ddot{y}_{\text{d}}(t), g-\ddot{z}_{\text{d}}(t))\big)\\
    {u_1}_0(t)  &= \norm{\subscr{\avec}{d}(t)}{} = \big( (g-\ddot{z}_{\text{d}}(t))^2 + \ddot{y}_{\text{d}}^2(t) \big)^\half\\
    {u_2}_0(t) &= \ddot{\varphi}_0(t).
  \end{split}
  \label{eq:lift_pvtol0}
\end{equation}

Equations in \eqref{eq:lift_pvtol0} allow to compute a trajectory of the
$\PVTOLo$ for a given output curve satisfying the annulus assumption, thus
exactly solving the trajectory lifting task (on any interval $[0,T]$). This is a
straightforward consequence of the input-output linearizability of $\PVTOLo$.

\subsection{Trajectory lifting for the coupled PVTOL via dichotomy}
Next, we prove that, given a desired output curve satisfying the annulus
assumption, it exists a trajectory of the PVTOL exactly solving the lifting task
for suitable positive values of the parameter $\EpsVTOL$ and, as $\EpsVTOL$ goes
to zero, this trajectory depends continuously on it. To prove the result we use
a feedback transformation that takes the system into the form of a \emph{driven
  pushed pendulum}
where the parameter $\EpsVTOL$ plays the role of
the pendulum length. The result that we prove is based on and extends results in
\cite{JH-AS-RF:05} on finding upright trajectories of an inverted pendulum.

Let us consider the feedback transformation given by
\begin{equation}
  \label{eq:u1u2inv}
  \left [ \begin{array}{c}
      u_1\\
      \EpsVTOL u_2
    \end{array} \right ]
  =
  \left [ \begin{array}{rl}
      \sin\varphi &~  -\cos \varphi\\[1.2ex]
      -\cos\varphi &~  -\sin \varphi
    \end{array} \right ]
  \left ( \left [ \begin{array}{r}
        0\\
        -g
      \end{array} \right ] +
    \left [ \begin{array}{c}
        v_1\\
        v_2
      \end{array} \right ] \right ).
\end{equation}
The dynamics of the system becomes
\begin{equation}
  \label{eq:PVTOLfbk}
  \begin{array}{cll}
    \phantom{\EpsVTOL}\ddot y       & = & v_1\\
    \phantom{\EpsVTOL}\ddot z       & = & v_2\\
    \EpsVTOL\ddot \varphi & = & (g-v_2)\sin\varphi -v_1\cos\varphi.
  \end{array}
\end{equation}

We have written the dynamics in a form that is somehow unusual, since the
parameter $\EpsVTOL$ appears in the left hand side of the differential
equation. This form has the advantage that it is well defined even for
$\EpsVTOL=0$. In this case the model in \eqref{eq:PVTOLfbk} is defined by an
algebraic differential equation. Also, to be consistent, we have to consider as
control inputs $u_1$ and $\EpsVTOL u_2$ and think of the roll dynamics in
\eqref{eq:PVTOL} as $\EpsVTOL \ddot\varphi = \EpsVTOL u_2$.

\begin{remark}
  The feedback transformation highlights an important property of the coupled
  PVTOL ($\EpsVTOL \neq 0$), i.e., it has a well defined relative degree, $r =
  [2, 2]$, with respect to the output $(y,z)$. It is worth noting that for
  $\EpsVTOL>0$ the zero dynamics of the system is unstable (the driven pendulum
  is pushed and thus inverted), and therefore the system is \emph{non-minimum
    phase}. This is an interesting feature of the PVTOL that makes the lifting
  task more challenging.
\end{remark}

An important role in the study of the trajectory manifold of the PVTOL is played
by the ``quasi trajectory'' that (with some abuse of notation) we call
\emph{quasi-static trajectory}. It is a time parametrized curve built pretending
that, at each instant $t$, the roll angle assumes the equilibrium value obtained
if $\ddot y (t)$ and $\ddot z (t)$ were constant. By imposing $\ddot{\varphi} =
0$ in equation \eqref{eq:PVTOLfbk} we get
\begin{equation}
  \tan \varphi_{qs}(t) = \frac{\ddot y (t)}{(g-\ddot z(t))}.
\end{equation}
It is worth noting that the quasi-static roll trajectory does not depend on
$\EpsVTOL$ and coincides with the roll trajectory that we obtained for the
$\PVTOLo$ system. Also, in the driven pushed pendulum the quasi-static
trajectory produces an acceleration vector $\subscr{\avec}{d}$ aligned along the
pendulum axis.


A first straightforward but interesting result can be proven. Before stating the
proposition we need some more notation. Recall that in
Section~\ref{sec:PVTOL-model} we have written the PVTOL dynamics in state space
form and denoted $x$ and $u$ the state and the input of the system. Consistently
with that notation we let $x$ be the state of system in \eqref{eq:PVTOLfbk} and
$v := (v_1, v_2)$, so that $\dot x = \subscr{f}{pend}(x,v)$ with suitably
defined $\subscr{f}{pend}$.

\begin{proposition}
  Let $\subscr{\yout}{d}(\cdot)$ be a desired $\CC^4$ output curve on
  $(-\infty,+\infty)$ satisfying the annulus assumption. Let
  $(\subscr{x}{}^*(\cdot),\subscr{u}{}^*(\cdot))$ be the trajectory solving the
  exact lifting for the $\PVTOLo$ model in \eqref{eq:PVTOL0} according to
  \eqref{eq:lift_pvtol0} and
  $(\subscr{x}{pend}^*(\cdot),\subscr{v}{pend}^*(\cdot))$ the trajectory solving
  the exact lifting for the model in \eqref{eq:PVTOLfbk} with $\EpsVTOL=0$. Then
  \[
  \subscr{x}{}^*(\cdot) = \subscr{x}{pend}^*(\cdot)
  \]
  and
  \[
  u_1^*(\cdot) = (g-\subscr{\ddot{z}}{d}(\cdot)) \cos\varphi^*(\cdot) +
  \subscr{\ddot{y}}{d}(\cdot)\sin\varphi^*(\cdot),
  \]
  consistently with equation \eqref{eq:u1u2inv}.
\end{proposition}

%
%


With this feedback transformation in hand, the exact lifting task for the PVTOL
can be formulated as follows. Given a desired output curve
$\subscr{\yout}{d}(\cdot)$, find a bounded roll trajectory for the roll dynamics
\begin{equation}
  \label{eq:Rolldyn}
  \begin{array}{cll}
    \EpsVTOL\ddot \varphi & = & (g-\subscr{\ddot{z}}{d}(t))\sin\varphi
    -\subscr{\ddot{y}}{d}(t)\cos\varphi.
  \end{array}
\end{equation}

The proof of existence of $\varphi_{\EpsVTOL}(\cdot)$ and right continuity with
respect to $\EpsVTOL$ is based on the presence of a dichotomy in the
linearization of the dynamics of an inverted pendulum about the vertical
position. In~\cite{JH-AS-RF:05} a bounded trajectory of the inverted pendulum is
proven to exist as a fixed point of a contraction mapping. Here, we generalize
the result in \cite{JH-AS-RF:05} in the sense that we allow the acceleration of
the pivot point to lie in the entire annulus (not only on the horizontal axis
$\ddot{z}=0$) and we study the properties of the lifted trajectory when the
parameter $\EpsVTOL$ (the length of the pendulum) goes to
zero.

We rewrite the roll dynamics in \eqref{eq:Rolldyn} in the form
\begin{equation}
  \begin{array}{cll}
    \EpsVTOL \ddot{\varphi} & = & \subscr{a}{d}(t) \sin (\varphi - \varphi_{qs}(t)),
  \end{array}
  \label{eq:Rolldyn_qs}
\end{equation}
where $\subscr{a}{d}(t) = \norm{\subscr{\avec}{d}(t)}{}$.
%
%
Then we write it in terms of the error from the \emph{quasi-static} angle,
$\theta = \varphi - \subscr{\varphi}{qs}$, as
\begin{equation}
  \begin{array}{cll}
    \EpsVTOL\ddot{\theta} & = & \subscr{a}{d}(t) \sin\theta - \EpsVTOL\ddot{\varphi}_{\text{qs}}(t),
  \end{array}
  \label{eq:Thetadyn}
\end{equation}
and, in order to use dichotomy, in terms of its linearization about $\theta =
0$,
\begin{equation}
  \begin{array}{cll}
    \EpsVTOL\ddot{\theta} & = & \subscr{a}{d}(t) \theta - \subscr{a}{d}(t) \Big(\theta - \sin \theta + \EpsVTOL\frac{\subscr{\ddot{\varphi}}{qs}(t)}{\subscr{a}{d}(t)}\Big).
  \end{array}
  \label{eq:Thetadyn_dicho}
\end{equation}
Let us consider the linear time-varying system driven by a bounded external
input
\begin{equation*}
  \begin{array}{cll}
    \EpsVTOL \ddot \gamma & = & \subscr{a}{d}(t)\gamma - \subscr{a}{d}(t)\mu(t).
  \end{array}
  \label{eq:Thetadyn_lin}
\end{equation*}
In \cite{JH-AS-RF:05} it was proven that, for any $\EpsVTOL>0$, the undriven
system admits an exponential dichotomy and, therefore that, working in a
noncausal fashion, for any bounded input $\mu(\cdot)$ a bounded solution
$\gamma(\cdot)$ exists. We let $\map{\AA_{\EpsVTOL}}{L_\infty}{L_\infty}$ be the
linear map $\mu(\cdot) \mapsto \gamma(\cdot)$. The following holds.
\begin{lemma}[Theorem~5 in \cite{JH-AS-RF:05}]
  For any $\EpsVTOL>0$, $\AA_{\EpsVTOL}$ is a bounded linear operator with
  $\norm{\AA_{\EpsVTOL}}{}=1$.
\end{lemma}

Defining the nonlinear operator ${\NN_{\EpsVTOL}}$ as
\begin{equation*}
  \begin{array}{c}
    \theta \rightarrow \AA_{\EpsVTOL}\left [\theta - \sin \theta +
      \EpsVTOL\frac{\subscr{\ddot{\varphi}}{qs}(t)}{\subscr{a}{d}(t)} \right ]
    =: {\NN_{\EpsVTOL}}\left [\theta (\cdot) \right],
  \end{array}
\end{equation*}
Clearly, a bounded curve $\theta(\cdot)$ is a solution of
\eqref{eq:Thetadyn} if and only if it is a fixed point of ${\NN_{\EpsVTOL}}$,
i.e. $\theta(\cdot) = {\NN_{\EpsVTOL}}\left[\theta (\cdot) \right]$.

%
%
%

We are now ready to prove the main result in this section. The proof relies on
arguments of Theorem~8 in \cite{JH-AS-RF:05}.
\begin{theorem}
  \label{thm:phi_eps_continuity}
  Given a $\CC^4$ output curve $(y(\cdot),z(\cdot))$ on $(-\infty,+\infty)$,
  with $0<\subscr{a}{min}\leq \subscr{a}{d}(t)\leq\subscr{a}{max}$ for all $t$
  ($\ddot{\varphi}_{\text{qs}}(\cdot)$ bounded), then there exists an
  $\EpsVTOLo>0$,
  \[
  \EpsVTOLo =
  \frac{1}{\norm{\subscr{\ddot{\varphi}}{qs}(\cdot)/\subscr{a}{d}(\cdot)}{L_\infty}},
  \]
  such that for any $\EpsVTOL \in (0, \EpsVTOLo)$
  \begin{enumerate}
  \item there exists a bounded trajectory $\theta_{\EpsVTOL}(\cdot)$ of
    \eqref{eq:Thetadyn} so that $\varphi_{\EpsVTOL}(\cdot) = \varphi_{qs}(\cdot)
    + \theta_{\EpsVTOL}(\cdot)$ is a bounded trajectory of \eqref{eq:Rolldyn}
    whenever $\varphi_{qs}(\cdot)$ is bounded.  and
  \item $\norm{\varphi_{\EpsVTOL}(\cdot) - \varphi_{\text{qs}}(\cdot)}{L_\infty}\leq
    \sin^{-1}\frac{\EpsVTOL}{\EpsVTOLo}$.
  \end{enumerate}
\end{theorem}

\begin{proof}
  In order to prove existence of the trajectory, we show that there exists
  $\delta_0>0$ such that for any $\delta\in (0,\delta_0)$ the map
  $\NN_{\EpsVTOL}$ is a contraction on the invariant set $\overline{B}_\delta =
  \{\theta(\cdot)\in L_\infty |~ \norm{\theta(\cdot)}{L_\infty}\leq\delta\}$. The set
  $\overline{B}_\delta$ is invariant if for any $\delta$
  \[
  \norm{\NN_{\EpsVTOL}[\theta(\cdot)]}{L_\infty}\leq
  \norm{\AA_{\EpsVTOL}}{}\norm{\theta(\cdot) -
    \sin\theta(\cdot)}{L_\infty}+\EpsVTOL\norm{\frac{\subscr{\ddot{\varphi}}{qs}(\cdot)}{\subscr{a}{d}(\cdot)}}{L_\infty}\leq\delta.
  \]
  Recall that $\norm{\AA_{\EpsVTOL}}{}=1$ for all $\EpsVTOL>0$ and $f(\delta) =
  \delta - \sin\delta$ is monotonically increasing on $[0,\pi]$. Therefore,
  posing
  $\EpsVTOLo=\frac{1}{\norm{\subscr{\ddot{\varphi}}{qs}(\cdot)/\subscr{a}{d}(\cdot)}{L_\infty}}$,
  $\overline{B}_\delta$, $\delta\in[0,\pi]$, is invariant under $\NN_{\EpsVTOL}$
  if
  \[
  \frac{\EpsVTOL}{\EpsVTOLo} \leq \sin\delta.
  \]
  Now, the function $f(\cdot)$ is Lipschitz continuous on
  $[0,\delta_0]\subset[0,\pi]$ with Lipschitz constant $1-\cos\delta_0$,
  i.e. $|f(\delta_1) - f(\delta_2)|\leq (1-\cos\delta_0) |\delta_1 - \delta_2|$
  for all $\delta_1$, $\delta_2$ $\in [0,\delta_0]\subset[0,\pi]$. Therefore,
  choosing $\delta<\pi/2$ and such that $\sin\delta \geq
  \frac{\EpsVTOL}{\EpsVTOLo}$, we have
  \[
  \norm{\NN_{\EpsVTOL}[\theta_1(\cdot)]-\NN_{\EpsVTOL}[\theta_2(\cdot)]}{L_\infty}\leq
  \rho \norm{\theta_1(\cdot) - \theta_2(\cdot)}{L_\infty},
  \]
  with $\rho = 1-\cos\delta < 1$, so that $\NN_{\EpsVTOL}$ is a contraction on
  the invariant set $\overline{B}_\delta$. In particular, the minimal set is
  obtained for $\delta = \sin^{-1}\frac{\EpsVTOL}{\EpsVTOLo}$. This gives the
  bound
  \[
  \norm{\theta(\cdot)}{L_\infty} = \norm{\varphi_{\EpsVTOL}(\cdot) -
    \varphi_{\text{qs}}(\cdot)}{L_\infty}\leq \sin^{-1}\frac{\EpsVTOL}{\EpsVTOLo}
  \]
  proving statement (ii).
\end{proof}

%
%

Next theorem follows easily from the results above.
\begin{theorem}[Exact lifting for the PVTOL]
  \label{thm:PVTOL_uncstr_lift}
  Let $\subscr{\yout}{d}(t)$, $t\in[0, T]$, be a desired $\CC^4$ output
  curve for the PVTOL. Then there exists a trajectory $(x^*(\cdot),
  u^*(\cdot))$, with initial condition $x_0=x^*(0)$, such that $p(x^*(t)) =
  \subscr{\yout}{d}(t)$ for all $t\in[0,T]$.
\end{theorem}

\begin{proof}
  The external velocities are easily the derivatives of the desired outputs,
  while the roll and roll rate trajectories $\phi^*(\cdot)$ and
  $\dot{\phi}^*(\cdot)$ on $[0,T]$ can be chosen as the restriction to $[0,T]$
  of the trajectories on the infinite horizon. These are proven to exist by
  Theorem~\ref{thm:phi_eps_continuity}.
\end{proof}

\section{Optimal control based strategy for feasible trajectory lifting}
\label{sec:explor_strategy}
In this section we provide a strategy to solve the practical feasible trajectory
lifting task.  The idea is to attack the problem by means of optimal control
combined with continuation and relaxation methods.

Let $\xi = (x(t), u(t))$, $t\in [0,T]$, be a state-input
curve. If $(x(t), u(t)) \in \cXU$ for almost all $t\in[0,T]$, we say that $\xi$
is a feasible curve and write $\xi\in\cXU$.
Proceeding formally, to solve the practical feasible lifting task in
Definition~\ref{def:feas_traj_lift}, we should solve the following optimal
control problem
\[
\begin{split}
  \min_{\xi\in\TT} & \; \norm{p(x(\cdot)) - \subscr{\yout}{d}(\cdot)}{L_2}\\
  \subj & \; \xi \in \cXU.
\end{split}
\]
The main idea behind our strategy is not to attack the constrained lifting
problem directly, but to embed it into a family of relaxed problems and use
continuation with respect to parameters to find a solution.
Informally, we perform the following steps: (i) embed the maneuvering system
into a family of systems, (ii) design a routine to solve the unconstrained
lifting task for a fixed embedding system in order to obtain a desired
(state-input) trajectory $\subscr{\xi}{d}$, (iii) minimize a weighted $L_2$
distance from the infeasible unconstrained lifted trajectory, $\subscr{\xi}{d}$,
over the feasible trajectories, (iv) relax the constrained optimal control
problem and design a solver for the relaxed problem (for each fixed embedding
system), and (v) design an update policy for the system and problem parameters.


\emph{Embedding systems}\\
We embed the maneuvering system into a family of systems with the property that,
for some choice of the parameters, the system has a ``special'' structure. That
is, there exist values of the parameters for which the (unconstrained) lifting
task can be solved more easily (e.g. because the system is differentially flat
or has some ``nice'' geometry).  As regards the PVTOL, we consider the family of
PVTOL models parametrized by the coupling parameter $\EpsVTOL$ with
$\EpsVTOL\geq0$. From the results in the previous section we know that we can
easily solve the unconstrained lifting task for the decoupled system, i.e. for
$\EpsVTOL=0$, and that an unconstrained trajectory is proved to exist for some
positive values of the parameter.

\emph{Unconstrained lifting}\\ 
The objective of this step is to obtain a (state-input) trajectory solving the
unconstrained lifting task to use as the desired curve for the constrained
problem. We use a \emph{dynamic embedding} technique introduced in
\cite{JH-AS-RF:04}. Informally, it consists of embedding the original system
into a fully controllable system and solve the practical trajectory lifting by
penalizing the \emph{embedding input} much more than the real ones. We describe
an ad-hoc version of this technique for the PVTOL, but it can be easily
generalized to other maneuvering systems.

First, observe that for $\EpsVTOL=0$ a lifted trajectory of the PVTOL can be
easily obtained by equations in \eqref{eq:lift_pvtol0}.
For $\EpsVTOL>0$ we know by Theorem~\ref{thm:phi_eps_continuity} that, under
suitable conditions on $\subscr{\yout}{d}(\cdot)$, there exists a trajectory (on
the infinite time horizon) solving the exact lifting task, and it depends
continuously on $\EpsVTOL$. 
An approximation of this trajectory on the finite horizon can be computed in
three steps: (i) compute the external position and velocities trajectories, (ii)
compute the internal trajectories, (iii) compute the input trajectories. The
first step is straightforward. The external position trajectories are simply
given by the desired curves and the velocities are obtained by
differentiation. The input trajectories can be easily computed by
equation~\eqref{eq:u1u2inv} once the trajectory of the internal position state
(the roll trajectory) is known. Thus, the dynamic embedding technique is
applied to compute the roll and roll rate trajectories. We use the dynamics in
the new coordinates \eqref{eq:PVTOLfbk} and find a trajectory of the
reduced system in \eqref{eq:Rolldyn}.
We embed the roll dynamics into the driven system
\begin{equation}
  \EpsVTOL \ddot \varphi = (g-\subscr{\ddot{z}}{d}(t))\sin \varphi
  -\subscr{\ddot{y}}{d}(t) \cos \varphi + \EpsVTOL \subscr{u}{emb},
  \label{eq:Rolldyn_EXT}
\end{equation}
where $\subscr{u}{emb}$ is the embedding input used to drive the system along
any desired admissible trajectory. 
If we rewrite \eqref{eq:Rolldyn_EXT} in state space form as $\dot x_{\varphi} =
f_{\varphi}(t, x_\varphi, u_{ext})$, where $x_{\varphi}=(\varphi, \dot \varphi)$
and $x_{qs}=(\varphi_{qs}, \dot \varphi_{qs})$, the following optimization
problem may be posed
\begin{equation}
  \begin{array}{l}
    \minimize~
    \frac{1}{2}\int_{0}^{T}\|x_\varphi(\tau) - x_{qs}(\tau)\|
    ^{2}_{Q_\varphi} + \rr |\subscr{u}{emb}(\tau)|^{2} d\tau+\frac{1}{2}
    \|x_\varphi(T) - x_{qs}(T) \| ^{2}_{P_\varphi}\\
    \text{subject to} ~ \dot x_\varphi = f_\varphi(t, x_\varphi,
    \subscr{u}{emb}), \qquad x_{\varphi}(0)=x_{qs}(0).
  \end{array}
  \label{eq:optim_dyn_embedding}
\end{equation}
where we use the \emph{quasi-static} trajectory as
a desired curve to find the actual trajectory,
%
%
${Q_\varphi}>0$ and ${P_\varphi}>0$ are positive definite weighting
matrices and $\rr >0$ is the weight of the embedding input. Using a sufficiently
high weight $\rr$ for the embedding input, we obtain a trajectory arbitrarily
close to the exact lifted trajectory. The optimization problem is solved by
using the projection operator based Newton method described in
Appendix~\ref{sec:prelim_proj_oper}.


We denote $\lift_{\EpsVTOL}$ the routine described above to solve the
unconstrained lifting.
Specifically, we let $\xi = \lift_{\EpsVTOL}(\xi_0; \subscr{\yout}{d}(\cdot))$
be a trajectory solving the practical lifting task for the desired output curve
$\subscr{\yout}{d}(\cdot)$ and computed by using $\xi_0$ as initial guess. The
routine is parametrized by the coupling parameter $\EpsVTOL$ of the PVTOL
dynamics.
For $\EpsVTOL=0$ the lifting procedure does not need any initial guess. Thus, we
simply write $\xi =\lift_{0}(\subscr{\yout}{d}(\cdot))$.

\emph{Constrained $L_2$ distance minimization and optimal control relaxation}\\
With a full unconstrained trajectory $\subscr{\xi}{d}=(\subscr{x}{d}(\cdot),
\subscr{u}{d}(\cdot))$ in hand, we can pose the following constrained optimal control
problem, where we minimize a weigthed $L_2$ distance from $\subscr{\xi}{d}$
subject to feasibility,
\begin{equation*}
  \begin{split}
    \min_{(x(\cdot), u(\cdot))} &\; \half \int_0^T \norm{x(\tau) -
      \subscr{x}{d}(\tau)}{Q}^2 + \half \norm{u(\tau) -
      \subscr{u}{d}(\tau)}{R}^2 \; d\tau + \half \norm{x(T) -
      \subscr{x}{d}(T)}{P_f}^2\\[1.2ex]
    \subj      &\; \dot{x}(t) = f(x(t), u(t)), \quad x(0)=x_0\\[1.2ex]
    &\; (x(t), u(t))\in\cXU, \qquad \text{for a.a. } t\in[0,T]
  \end{split}
\end{equation*}
where $Q$, $R$ and $P_f$ are positive definite matrices. The idea is to choose
$Q$, $R$ and $P_f$ so that the weights associated to the outputs (external
position states) are much larger than the weights of the other sates and the
inputs. 
Denoting $\norm{\xi-\xi_d}{L_2} := \half \int_0^T \norm{x(\tau) -
  \subscr{x}{d}(\tau)}{Q}^2 + \norm{u(\tau) - \subscr{u}{d}(\tau)}{R}^2 \; d\tau
+ \half \norm{x(T) - \subscr{x}{d}(T)}{P_f}^2$, we can rewrite the problem in a
more compact notation as
\begin{equation}
  \begin{split}
    \min_{\xi\in\TT_{\EpsVTOL}} & \; \norm{\xi-\xi_d}{L_2}\\
    \subj & \; \xi \in \cXU,
  \end{split}
  \label{eq:constr_min}
\end{equation}
where we have denoted $\TT_{\EpsVTOL}$ the trajectory manifold for a given value
of the parameter $\EpsVTOL$.



Now, we introduce a relaxed version of the above optimal control problem. To do
that, we first define a relaxed version of the feasibility region. That is, we
parametrize the feasibility region by means of a scaling factor $\rhoFR\in[0,1]$
that allows to enlarge the nominal region up to a larger one containing the
unconstrained lifted trajectory $\xi_d$.
\begin{definition}[Scalable feasibility region]
  A scalable feasibility region is defined as
  \[
  \cXU_\rhoFR = \{ (x,u) \in \real^n \times \real^m |~ c_j(x,u; \rhoFR) \leq 0,
  \rhoFR\in[0,1], ~ j \in\until{k}\}
  \]
  such that
  \begin{enumerate}
  \item $c_j(x,u; \rhoFR)$, $j\in\until{k}$, is $\CC^2$ in $x$ and $u$ and
    varies smoothly with $\rhoFR$.
  \item for any $\rhoFR\in[0,1]$, $\XU_\rhoFR$, the interior of $\cXU_\rhoFR$,
    is a nonempty simply connected set.
  \item for any $0<\rhoFR_1 < \rhoFR_2$, $\XU_{\rhoFR_1} \supset
    \XU_{\rhoFR_2}$;
  \item the projection of $\XU_{\rhoFR}$ on the input space
    $\pi_u \XU_\rhoFR = \{u \in \real^m | (x, u) \in \XU_\rhoFR
    ~\forall~\text{fixed}~ x\}$
    is convex;
  \item for every desired output trajectory $\subscr{\yout}{d}(\cdot)$, the
    lifted trajectory $\subscr{\xi}{d} = (\subscr{x}{d}(\cdot),
    \subscr{u}{d}(\cdot))$ (if it exists) is such that $\exists~
    \rhoFR_0\in[0,1]$ such that $(\subscr{x}{d}(t), \subscr{u}{d}(t)) \in
    \XU_{\rhoFR_0}$ for every $t \in [0,T]$.   
  \end{enumerate}
\end{definition}

With this definition in hand we can introduce the relaxed version of the optimal
control problem in \eqref{eq:constr_min}. We use the barrier functional idea
described in Appendix~\ref{sec:prelim_proj_oper}. Namely, we add to the cost
functional a barrier term to enforce feasibility with respect to the point-wise
constraints. Thus, the optimal control problem relaxation is given by
\begin{equation}
  \begin{split}
    \min_{\xi \in \TT_{\EpsVTOL}} \;
    \norm{\xi-\xi_d}{L_2} + \EpsBeta b_{\DelBeta, \rhoFR}(\xi)
  \end{split}
  \label{eq:opt_contr_relax_pvtol}
\end{equation}
where $\TT_{\EpsVTOL}$ is the trajectory manifold for a given value of
$\EpsVTOL$ and $b_{\DelBeta, \rhoFR}(\xi)$ is a barrier functional defined
consistently with \eqref{eq:barrier-functional}. It is worth noting that here
the barrier functional is parametrized also by $\rhoFR$ because the constraints
are.
%
The relaxed optimal control problem is solved by using the projection operator
Newton method in Appendix~\ref{sec:prelim_proj_oper}.

Next, we introduce some useful notation. For a given scalable feasibility region
$\XU_\rhoFR$, parametrized by $\rhoFR$, we denote
$\ponewt_{\EpsVTOL}(\subscr{\xi}{d}, \xi_0; \EpsBeta, \rhoFR)$ a routine that
takes as inputs a desired (state-control) curve $\subscr{\xi}{d}$ and an initial
trajectory $\xi_0$, and computes a feasible trajectory $\xi \in\TT_{\EpsVTOL}$,
with $\xi\in\XU_\rhoFR$, by solving the nonlinear optimal control relaxation in
\eqref{eq:opt_contr_relax_pvtol}. The routine is parametrized by the embedding
parameter $\EpsVTOL$ (the coupling parameter for the PVTOL), the parameter
$\EpsBeta$ scaling the barrier functional (see
Appendix~\ref{sec:barrier_functional}) and the parameter $\rhoFR$ scaling the
feasibility region.


We are now ready to define our strategy. First, we provide an informal
description. From now on we call \emph{unconstrained lifted trajectory} a
trajectory solving the practical (unconstrained) lifting task.

\begin{quote}
  \emph{Lift and Constrain Strategy}: The strategy consists of the following
  steps: (i) given a desired output curve $\subscr{\yout}{d}(\cdot)$, an
  unconstrained lifted trajectory for an initial embedding system is computed
  (e.g., for the decoupled PVTOL with $\EpsVTOL=0$); (ii) a continuation update
  on the parameter $\EpsVTOL$ is applied up to the nominal value
  $\EpsVTOL^{\text{nom}}$; (iii) a relaxed optimal control problem is solved for
  $\rhoFR=0$ (feasibility region containing the lifted trajectory) and $\EpsBeta
  = {\EpsBeta}_0$ (for a suitable ${\EpsBeta}_0 >0$); (iv) a continuation update
  on the parameter $\rhoFR$ is applied shrinking the feasibility region up to
  its nominal value for $\rhoFR=1$; (v) a continuation update on the parameter
  $\EpsBeta$ is applied to regulate the closeness of the feasible trajectory
  from the boundary ($\EpsBeta\rightarrow {\EpsBeta}_e$ for some ${\EpsBeta}_e$).
\end{quote}

Next we give a pseudo-code description of the strategy.

\begin{center}
  \begin{minipage}{0.5\linewidth}

    \medskip \hrule width \linewidth \smallskip

    \noindent\begin{minipage}{0.44\linewidth}\textbf{\texttt{Strategy:}}%
    \end{minipage}%
    \begin{minipage}{0.56\linewidth} Lift and Constrain Strategy
    \end{minipage}\\[1.2ex]
    \noindent\begin{minipage}{0.44\linewidth}\textbf{\texttt{Task:}}%
    \end{minipage}%
    \begin{minipage}{0.56\linewidth} Feasible trajectory lifting%
    \end{minipage}\\[1.2ex]
    \noindent\begin{minipage}{0.44\linewidth}\textbf{\texttt{Inputs:}}%
    \end{minipage}%
    \begin{minipage}{0.56\linewidth}
      $\EpsVTOL^{\text{nom}}$, $\cXU$, $\subscr{\yout}{d}(\cdot)$, ${\EpsBeta}_e$%
    \end{minipage}\\[1.2ex]
    \noindent\begin{minipage}{0.44\linewidth}\textbf{\texttt{Output:}}%
    \end{minipage}%
    \begin{minipage}{0.56\linewidth}
      $\subscr{\xi}{c}\in\TT_{\EpsVTOL^{\text{nom}}}$ \; \text{with} \; $\subscr{\xi}{c}\in\XU$%
    \end{minipage}\\[1.2ex]

    \noindent\begin{minipage}{0.44\linewidth}\textbf{\texttt{Parameters:}}\\[1.2ex]%
    \end{minipage}%
    \begin{minipage}{0.56\linewidth}
      $(\EpsVTOL, \EpsBeta, \rhoFR) \in\real^3_{\geq0}$\\[1.2ex]
      $(\subscr{\xi}{d}, \subscr{\xi}{c}) \in\TT_{\EpsVTOL}\times\TT_{\EpsVTOL}$%
    \end{minipage}\\[1.2ex]

    \noindent\begin{minipage}{0.44\linewidth}\textbf{\texttt{Initialization:}}\\[1.2ex]%
      ~\\[1.2ex]
    \end{minipage}%
    \begin{minipage}{0.56\linewidth}
      $\EpsVTOL :=0$, $\EpsBeta :={\EpsBeta}_0$, $\rhoFR :=0$\\[1.2ex]
      $\subscr{\xi}{d} :=\lift_0(\subscr{\yout}{d}(\cdot))$\\[1.2ex]%
      $\subscr{\xi}{c} := \subscr{\xi}{d}$
    \end{minipage}

    \medskip

    \begin{itemize}
    \item[1.] \textbf{\texttt{WHILE}} $\EpsVTOL < \EpsVTOL^{\text{nom}}$
      \textbf{\texttt{DO}}

      increase $\EpsVTOL$

      $\subscr{\xi}{d} = \lift_{\EpsVTOL}(\subscr{\xi}{d};
      \subscr{\yout}{d}(\cdot))$\\%
      \textbf{\texttt{END}}
    \item[2.] \textbf{\texttt{WHILE}} $\rhoFR < 1$ \textbf{\texttt{DO}}

      increase $\rhoFR$


      $\subscr{\xi}{c} = \ponewt_{\EpsVTOL}(\subscr{\xi}{d}, \subscr{\xi}{c};
      \EpsBeta, \rhoFR)\\$
      \textbf{\texttt{END}}
    \item[3.] \textbf{\texttt{WHILE}} $\EpsBeta >
      {\EpsBeta}_e$ 
      \textbf{\texttt{DO}}

      decrease $\EpsBeta$


      $\subscr{\xi}{c} = \ponewt_{\EpsVTOL}(\subscr{\xi}{d}, \subscr{\xi}{c};
      \EpsBeta, \rhoFR)\\$
      \textbf{\texttt{END}}

    \item[4.] \textbf{\texttt{RETURN}} $\subscr{\xi}{c}$
    \end{itemize}

    \medskip \hrule width \linewidth \smallskip

  \end{minipage}

\end{center}

\begin{remark}[Variations of the strategy]
  Other variations of the above strategy can be obtained by changing the order
  of or combining the update steps of some parameters in the continuation
  strategy. These different choices can be thought as degrees of freedom in the
  designer's hands and their effectiveness is strongly related to the system
  dynamics and to the feasibility constraints.
\end{remark}


\section{Strategy analysis}
\label{sec:strategy-analysis}
In this section we prove that under suitable conditions on the feasibility
region we can find a feasible trajectory that solves locally the practical
lifting task in Definition~\ref{def:pract_feas_traj_lift}.

\subsection{Differentiability of an optimal control minimizer with respect to
  parameters}
\label{sec:chapTE-existence-results}%
We start providing a supporting result to prove the existence of a feasible
trajectory. Namely, we prove, under suitable conditions, continuity and
differentiability of an optimal control minimizer with respect to parameters.
We present this result in a separate subsection for two reasons. First, this is
the most subtle part to prove the existence of a feasible trajectory. Second,
we believe this is an important stand alone result.



We consider an optimal control problem where the cost functional depends
smoothly on a finite dimensional parameter and the system is independent of the
parameter. In this section, we will refer to this parameter as
$\rho\in\real^p$. This parameter may include, for instance, the scalar parameter
$\rhoFR$ used for specifying the size of the feasible region as well as the
scalar parameter $\EpsBeta$ used in determining strictly feasible trajectories.
We thus write the minimization problem
\[
\min_{\xi\in\TT} h(\xi,\rho).
\]
Using Lemma~\ref{lem:min_via_projection} in Appendix~\ref{sec:prelim_proj_oper},
we can look for an unconstrained local minimum of the functional
%
\[
g_\rho(\xi) = h(\PP(\xi),\rho).
\]
We will suppose that the scaling and offset of the parameters have been chosen
in such a manner that the nominal value of the parameter vector is $\rho=0$.
%
%
%

\begin{remark}[Parametrization with respect to system parameters]
  The parameter $\rho$ does not include the system parameter $\EpsVTOL$. Indeed,
  the parameter $\EpsVTOL$ affects the dynamics and, thus, the projection
  operator $\PP$. The following results hold true for the case where also the
  projection operator depends on the parameter,
  i.e. $g_\rho(\xi)=h(\PP(\xi,\rho),\rho)$, provided the projection operator is
  shown to depend smoothly on the system parameter.
\end{remark}

Let $\xi_0\in\TT$ and consider the nature of $g_0(\xi)$ on a neighborhood of
$\xi_0$. In particular, we consider $\xi$ of the form $\xi_0+\zeta$ where
$\|\zeta\|<\delta$ and $\delta>0$ is such that $\xi_0+\zeta\in \text{dom}\;\PP$
for each such $\zeta$. For $\CC^2$ $g_0(\cdot)$, we have
\begin{equation}
  g_0(\xi_0+\zeta)
  =
  g_0(\xi_0)
  + Dg_0(\xi_0)\cdot\zeta
  + \frac{1}{2} D^2g_0(\xi_0)\cdot(\zeta,\zeta)
  + r(\xi_0,\zeta)\cdot(\zeta,\zeta)
  \label{eq:expand_g}
\end{equation}
where the remainder satisfies
\begin{equation}
  | r(\xi_0,\zeta)\cdot(\zeta,\zeta) | / \|\zeta\|^2 \to 0
  \text{~~as~~}
  \|\zeta\| \to 0
  \label{eq:remainder_linf}
\end{equation}
where $\|\cdot\|$ is the $L_\infty$ norm. Using the $\CC^2$ identity
\[
\phi(1) = \phi(0) + \phi'(0) + \int_0^1 (1-s)\,\phi''(s)\, ds
\]
together with $\phi(s) = g_0(\xi_0 + s\zeta)$, we obtain the explicit expression
\begin{equation}
  r(\xi_0,\zeta)\cdot(\zeta_1,\zeta_2)
  =
  \int_0^1
  (1-s)
  \left[
    D^2g_0(\xi_0+s\zeta) - D^2g_0(\xi_0)
  \right]
  ds \cdot (\zeta_1,\zeta_2)
  \label{eq:r_z_z}
\end{equation}
which has been slightly generalized to depend on three, possibly independent,
perturbations. Using the fact that $D^2g_0(\cdot)$ is continuous as a mapping
from the trajectory manifold $\TT$ to set of continuous bilinear functionals on
$L_\infty$, we easily verify that the remainder
$r(\xi_0,\zeta)\cdot(\zeta,\zeta)$ defined by (\ref{eq:r_z_z}) satisfies, as it
\emph{must}, the higher order property (\ref{eq:remainder_linf}). Equations
(\ref{eq:expand_g}), (\ref{eq:r_z_z}) provide a \emph{second order expansion
  with remainder} formula for the $\CC^2$ mapping $g_0(\cdot)$, valid in an
$L_\infty$ neighborhood of any $\xi_0\in\TT$.
In fact, the formula given by (\ref{eq:expand_g}), (\ref{eq:r_z_z}) is somewhat
more general, requiring only that $\xi_0\in L_\infty$ and $\delta>0$ are such
that $B_\delta(\xi_0)\subset \text{dom}\;\PP$.

Now, since the functional $g_0(\cdot)$ is the composition of an integral
functional and a projection operator, from \cite{JH:02} the value of the bilinear expression
$D^2g_0(\xi)\cdot(\zeta_1,\zeta_2)$ for $\xi\in \text{dom}\;\PP$ and $\zeta_i\in
L_\infty$ is of the form
\[
D^2g_0(\xi)\cdot(\zeta_1,\zeta_2) = \int_0^T \gamma_1(\tau)^T W(\tau)
\gamma_2(\tau) \; d\tau + (\pi_1\gamma_1(T))^T P_f (\pi_1\gamma_2(T)),
\]
where $\gamma_i = D\PP(\xi)\cdot\zeta_i$ and where $P_f = P_f^T\in\real^{n\times
  n}$ and the bounded matrix $W(t) = W(t)^T\in\real^{n\times n}$, $t\in[0,T]$,
depend continuously on $\eta = \PP(\xi)$, hence continuously on $\xi$. Using
these facts, we see that
\begin{lemma}
  Let $\xi_0\in\TT$ and suppose that $\delta>0$ is such that
  $B_\delta(\xi_0)\subset\text{dom}\;\PP$.  Then, there is a nondecreasing
  function $\bar{r}(\cdot)$ with $\bar{r}(0)=0$ such that
  \begin{equation}
    | r(\xi_0,\zeta)\cdot(\zeta_1,\zeta_2) |
    \le
    \bar{r}(\|\zeta\|) \; \|\zeta_1\|_{L_2} \|\zeta_2\|_{L_2}
    \label{eq:remainderL2bound}
  \end{equation}
  for all $\zeta,\zeta_1,\zeta_2\in B_\delta$.
  \label{lmm:remainderL2bound}
\end{lemma}
\begin{proof}
  Since $\PP$ is $\CC^2$, $D\PP(\xi)$ is a continuous linear projection operator
  with respect to the $L_\infty$ norm.  Using an explicit formula for
  $D\PP(\xi)\cdot\zeta$, it can be shown, \cite{JH:02}, that 
  $D\PP(\xi)$ may be extended to a linear projection operator $D\PP(\xi)_{L_2}$
  on $L_2$ that is continuous with respect to the $L_2$ norm.  The result
  follows easily using (\ref{eq:r_z_z}).
\end{proof}


Suppose now that $\xi_0\in\TT$ is a stationary trajectory of $g_0(\cdot)$ so
that $Dg_0(\xi_0)\cdot\zeta=0$ for all $\zeta\in L_\infty$ and that $\delta>0$
is such that $B_\delta(\xi_0)\subset\text{dom}\;\PP$. It follows that
\begin{equation}
  g_0(\xi_0+\zeta)
  \ge
  g_0(\xi_0) + \frac{1}{2} D^2g_0(\xi_0)\cdot(\zeta,\zeta)
  - \bar{r}(\|\zeta\|)\; \|\zeta\|^2_{L_2}
  \label{eq:g_xi_zeta_lower_bound}
\end{equation}
for all $\|\zeta\|<\delta$ where $\bar{r}(\cdot)$ is given by
Lemma~\ref{lmm:remainderL2bound}. Restricting (\ref{eq:g_xi_zeta_lower_bound})
to $\zeta\in T_{\xi_0}\TT$, we obtain the fundamental second order sufficient
condition (SSC) for $\xi_0$ to be an isolated local minimizer.

\begin{theorem}
  Suppose $\xi_0\in\TT$ is such that $Dg_0(\xi_0)\cdot\zeta=0$ for all $\zeta\in
  L_\infty$ and that there is a $c_0>0$ such that
  \begin{equation}
    D^2g_0(\xi_0)\cdot(\zeta,\zeta) \ge c_0\|\zeta\|^2_{L_2}
    \text{~~for all~~}   \zeta\in T_{\xi_0}\TT  \, .
    \label{eq:SSC}
  \end{equation}
  Then $\xi_0$ is an isolated local minimizer in the sense that there is a
  $\delta>0$ such that
  \[
  g_0(\xi_0) < g_0(\xi)
  \]
  for all $\xi\in\TT$ with $\|\xi-\xi_0\|<\delta$, $\xi\neq\xi_0$.
  \label{thm:SSC}
\end{theorem}
\begin{proof}
  Taking $\delta_1>0$ be such that $\bar{r}(\delta_1) < c_0/4$, we find that
  $g_0(\xi_0 + \zeta) \geq g_0(\xi_0) + (c_0/4) \|\zeta\|^2_{L_2}$ for all
  $\zeta\in T_{\xi_0}\TT$ with $\|\zeta\|<\delta_1$.  By
  Theorem~\ref{thm:traj_man_repr}, each $\xi\in\TT$ near $\xi_0$ can be
  represented by a unique $\zeta\in T_{\xi_0}\TT$ according to $\xi =
  \PP(\xi_0+\zeta)$ and the mapping $\xi\mapsto\zeta$ is continuous.  Thus there
  is a $\delta<\delta_1$ such that $\|\xi-\xi_0\|<\delta$ implies that
  $\|\zeta\|<\delta_1$.  The result follows.
\end{proof}
\noindent We call a local minimizer $\xi_0\in\TT$ satisfying (\ref{eq:SSC}) a
\emph{second order sufficient condition local minimizer}, SSC local minimizer
for short. According to Theorem~\ref{thm:SSC}, \emph{every} SSC local minimzer
is an \emph{isolated} local minimizer. We also note that, in words, the
condition (\ref{eq:SSC}) says that the quadratic functional $\zeta\mapsto
D^2g_0(\xi_0)\cdot(\zeta,\zeta)$ is \emph{strongly positive} on the subspace
$T_{\xi_0}\TT$.

Consider now the (local) minimization of $g_\rho(\xi)$ as the parameter $\rho$
is varied on a neighborhood of $\rho=0$ where $\xi_0$ is known to be an SSC
local minimizer of $g_0(\xi)$. Since $D^2g_\rho(\xi)$ is continuous in both
$\xi$ and $\rho$, we expect that, for each sufficiently small $\rho$, there will
be a corresponding SSC local minimizer $\xi_\rho$ near $\xi_0$ and that the
mapping $\rho\mapsto\xi_\rho$ will be continuous, and perhaps
differentiable. The key idea is to use an appropriate implicit function theorem
(IFT) to \emph{solve} the first order necessary condition equation
\begin{equation}
  Dg_\rho(\xi_\rho) = 0
  \label{eq:FNC}
\end{equation}
for $\xi_\rho$ as a function of $\rho$ starting from $\xi_0$ at
$\rho=0$. Proceeding formally, we differentiate (\ref{eq:FNC}) with respect to
$\rho$ to obtain
\[
\frac{\partial}{\partial \rho} Dg_\rho(\xi_\rho) + D\{Dg_\rho(\xi_\rho)\}\cdot
\xi_\rho' = 0 \, .
\]
Thus, the derivative of $\xi_\rho$ with respect to $\rho$, $\xi_\rho'$, if it
exists, is given formally by
\[
\xi_\rho' = - [ D\{Dg_\rho(\xi_\rho)\} ]^{-1} \cdot \frac{\partial}{\partial
  \rho} Dg_\rho(\xi_\rho) \, .
\]
In this case, we 
expect that there is an implicit function theorem that says something like, if
$D\{Dg_\rho(\xi_\rho)\}$ is invertible at $\rho=0$, then there is a neighborhood
of $\rho=0$ on which $\rho\mapsto\xi_\rho$ is well defined and $\CC^1$. In what
sense should the operator $D\{Dg_0(\xi_0)\}$ be invertible and how can it be
ensured? It turns out that the appropriate condition is that $D^2g_0(\xi_0)$ be
strongly positive on $T_{\xi_0}\TT$, i.e., that it satisfy (\ref{eq:SSC}).

\begin{theorem}
  Suppose that $\xi_0\in\TT$ is an SSC local minimizer of $g_0(\xi)$.  Then,
  there is a $\delta>0$ such that, for each $\rho$ such that $\|\rho\|<\delta$,
  there is a local SSC minimizer $\xi_\rho$ of $g_\rho(\xi)$ near $\xi_0$.
  Furthermore $\rho\mapsto\xi_\rho$ is continuously differentiable.
  \label{thm:cont_minimizer}
\end{theorem}
\begin{proof}
  The key is to show that, for $\rho$ sufficiently small, we can compute a
  $\xi\in\TT$ such that $Dg_\rho(\xi)\cdot\zeta = 0$ for all $\zeta\in
  L_\infty$. Using again Theorem~\ref{thm:traj_man_repr}, we proceed by
  parametrizing $\xi\in\TT$ locally by $\gamma\in T_{\xi_0}\TT$ according to
  $\xi = \PP(\xi_0+\gamma)$ and searching over $\gamma$.  As in the proof of
  many IFTs, we solve for the desired $\gamma$ using a contraction mapping.  For
  simplicity, we will denote $T_{\xi_0}\TT$ by $X$ so that we search for
  $\gamma\in X$ such that $Dg_\rho(\PP(\xi_0+\gamma))\cdot\zeta=0$ for all
  $\zeta\in L_\infty$.

  First, note that, since $D^2g_0(\xi_0)$ is strongly positive on $X$, the
  well-defined quadratic minimization problem
  \[
  \lambda = \text{arg}\;\min_{\zeta\in X} \, -\omega\cdot\zeta + \frac{1}{2}
  D^2g_0(\xi_0)\cdot(\zeta,\zeta)
  \]
  defines a linear mapping $\SS:\omega\mapsto\lambda:\text{dom}\;\SS\subset
  X^{*} \to X$ for some continuous linear functionals $\omega\in X^{*}$.  The
  linear mapping $\SS$ provides the solution $\lambda\in X$ to the functional
  equation
  \[
  D^2g_0(\xi_0)\cdot(\lambda,\zeta) = \omega\cdot\zeta, ~~ \zeta\in X,
  \]
  effectively providing an inverse to the operator $D\{Dg_0(\xi_0)\}$ formally
  described above.  We will see that the functionals $\omega\in X^{*}$ of
  interest belong to the domain of $\SS$.

  Define $\FF_\rho: X \to X^{*}$ by
  \[
  \FF_\rho(\gamma)\cdot\zeta = D^2g_0(\xi_0)\cdot(\gamma,\zeta) -
  Dg_\rho(\xi_0+\gamma)\cdot\zeta
  \]
  for all $\zeta\in X$.  Note that $\FF_\rho(\gamma)\cdot\zeta$ is of the form
  \[
  \FF_\rho(\gamma)\cdot\zeta = \int_0^T a(\tau)^T z(\tau) + b(\tau)^T v(\tau)\;
  d\tau + r_1^T z(T)
  \]
  for $\zeta=(z(\cdot),v(\cdot))\in X$ where $a(\cdot)$, $b(\cdot)$, and $r_1$
  depend smoothly on the data $\rho$ and $\gamma$.  It follows that
  $\FF_\rho(\gamma)\in\text{dom}\;\SS\subset X^{*}$.  A straightforward
  calculation shows that
  \[
  \GG_\rho(\gamma) = \SS \cdot \FF_\rho(\gamma)
  \]
  defines a continuous operator $\GG_\rho : X\to X$ that is also continuous in
  $\rho$.

  Note that, if $\gamma\in X$ is a fixed point of $\GG_\rho(\cdot)$,
  $\gamma=\GG_\rho(\gamma)$, then \mbox{$Dg_\rho(\xi_0+\gamma)\cdot\zeta=0$} for
  all $\zeta\in X$.  This will imply that $Dg_\rho(\xi_0+\gamma)\cdot\zeta=0$
  for all $\zeta\in L_\infty$ provided that $\PP(\xi_0+\gamma)$ is sufficiently
  near $\xi_0$.  In that case, we conclude that $\xi_\rho = \PP(\xi_0+\gamma)$.
  Also, for $\rho=0$, we see that $\gamma=0$ is the fixed point, $\GG_0(0) = 0$
  as expected.

  We will show that, for $\rho$ sufficiently small, $\GG_\rho(\cdot)$ is a
  contraction mapping with a unique fixed point.  For $\rho=0$, noting that
  $Dg_\rho(\xi_0+\gamma)\cdot(\cdot) = D^2g_0(\xi_0)\cdot(\gamma,\cdot) +
  o(\|\gamma\|)$, we see that
  \[
  \GG_0(\gamma)
  = \SS\cdot (D^2g_0(\xi_0)\cdot(\gamma,\cdot) -
  Dg_\rho(\xi_0+\gamma)\cdot(\cdot)) = o(\|\gamma\|)
  \]
  where we have used the fact that $\SS$ is continuous (bounded) on the elements
  of $X^{*}$ of the noted form.  By continuity in $\rho$, we see that there
  exist $\rho_1,\delta>0$ such that
  \[
  \| \GG_\rho(\gamma) \| \le \delta
  \]
  whenever $\|\rho\|\le\rho_1$ and $\|\gamma\|\le\delta$.  Now, fixing $\rho$,
  $\|\rho\|\le\rho_1$,
  \[
  \begin{array}{ccl}
    \GG_\rho(\gamma_1) - \GG_\rho(\gamma_2)
    & = &
    \SS \cdot
    \left[ Dg_\rho(\xi_0+\gamma_2)\cdot(\cdot)
      - Dg_\rho(\xi_0+\gamma_1)\cdot(\cdot) \right]
    \\[2ex]
    & = & \displaystyle
    \SS \cdot
    \left[
      \int_0^1 D^2g_\rho(\xi_0 + \gamma_1 + s(\gamma_2-\gamma_1)) \, ds
      \cdot (\gamma_2-\gamma_1,\cdot)
    \right]
  \end{array}
  \]
  so that there is a $k<\infty$ such that
  \[
  \| \GG_\rho(\gamma_1) - \GG_\rho(\gamma_2) \| \le k \delta \| \gamma_1 -
  \gamma_2 \|
  \]
  for $\|\gamma_1\|\le\delta$ and $\|\gamma_2\|\le\delta$.  Shrinking $\delta$,
  if necessary, so that $k\delta \le 1/2$, we see that $\GG_\rho$ is a
  contraction with unique fixed point $\gamma_\rho$.

  To see that $\rho \mapsto \gamma_\rho$ is continuous, write
  \[
  \begin{array}{ccl}
    \| \gamma_{\rho_1} - \gamma_{\rho_2} \|
    &  =  &
    \| \GG_{\rho_1}(\gamma_{\rho_1}) - \GG_{\rho_2}(\gamma_{\rho_2}) \| \\
    & \le &
    \| \GG_{\rho_1}(\gamma_{\rho_1}) - \GG_{\rho_1}(\gamma_{\rho_2}) \|
    +
    \| \GG_{\rho_1}(\gamma_{\rho_2}) - \GG_{\rho_2}(\gamma_{\rho_2}) \| \\
    & \le &
    (1/2)\| \gamma_{\rho_1} - \gamma_{\rho_2} \| +
    \| \GG_{\rho_1}(\gamma_{\rho_2}) - \GG_{\rho_2}(\gamma_{\rho_2}) \|
  \end{array}
  \]
  so that
  \[
  \| \gamma_{\rho_1} - \gamma_{\rho_2} \| \le 2 \| \GG_{\rho_1}(\gamma_{\rho_2})
  - \GG_{\rho_2}(\gamma_{\rho_2}) \|
  \]
  showing that $\rho\mapsto\gamma_\rho$ is continuous since
  $\rho\mapsto\GG_\rho(\gamma)$ is continuous (for fixed $\gamma$).

  Differentiability is proven following standard arguments from implicit
  function theorems applied to $\GG_\rho$. See, e.g., the second part of the
  proof of Theorem~4.E in \cite{EZ:95}.
\end{proof}

\subsection{Existence of a feasible lifted trajectory}
With the continuity result of the previous subsection in hands, we can prove our
existence result.
For a general maneuvering system we make the following standing assumption.
\begin{assumption}
  \label{ass:exist_lift_traj}
  Let $\subscr{\yout}{d}(t)$, $t\in[0,T]$, be a given desired output
  curve. There exists a (state-input) trajectory $\subscr{\xi}{d} =
  (\subscr{x}{d}(\cdot),\subscr{u}{d}(\cdot))\in\TT$ on $[0,T]$, such that
  $p(\subscr{x}{d}(t)) = \subscr{\yout}{d}(t)$ for all $t\in[0,T]$.
\end{assumption}

Formally, we state our main result in the next theorem.
\begin{theorem}[Existence of a feasible lifted trajectory]
  \label{thm:exist_feas_traj}
  Let $\subscr{\yout}{d}(t)$, $t\in[0, T]$, be a desired sufficiently
  smooth output curve satisfying Assumption~\ref{ass:exist_lift_traj} and
  $\cXU\subset\real^n\times\real^m$ a compact feasibility region. Then, the
  following holds
  \begin{enumerate}
  \item for a given $\DelBeta>0$, there exist $\rhoFR_0>0$ and ${\EpsBeta}_0>0$
    such that the problem
    \begin{equation}
      \min_{\xi\in\TT} \norm{\xi-\xi_d}{L_2} + \EpsBeta
      b_{\DelBeta,\rhoFR}(\xi), 
      \label{eq:opt_contr_relax}
    \end{equation}
    has an isolated local minimizer for all $0 \leq \rhoFR<\rhoFR_0$ and $0\leq
    \EpsBeta < {\EpsBeta}_0$;
  \item if $\xi^*$ is an isolated local minimizer of the problem in
    \eqref{eq:opt_contr_relax} for given $\rhoFR>0$ and ${\EpsBeta}>0$, then
    there exists $\epsilon_2>0$ such that
    \[
    \norm{\xi^*-\subscr{\xi}{d}}{L_2} \leq \norm{\xi-\subscr{\xi}{d}}{L_2} +
    \epsilon_2
    \]
    for all trajectories $\xi\in\TT$ in a neighborhood of $\xi^*$;
  \item for $\xi^*$ as in (ii), then there exists $\epsilon_3>0$ such that
    \[
    \norm{p(x^*(\cdot))-\subscr{\yout}{d}(\cdot)}{L_2} \leq
    \norm{p(x(\cdot))-\subscr{\yout}{d}(\cdot)}{L_2} + \epsilon_3
    \]
    for all trajectories $\xi=(x(\cdot), u(\cdot))\in\TT$ in a neighborhood of
    $\xi^*\in\TT$.
  \end{enumerate}
\end{theorem}

\begin{proof}
  Statement~(i) is just a straightforward corollary of
  Theorem~\ref{thm:cont_minimizer}.
  To prove statement~(ii), we observe that from (i) there exists $\delta>0$ such
  that
  \[
  \norm{\xi^*-\xi_d}{L_2} + \EpsBeta b_{\rhoFR}(\xi^*) \leq
  \norm{\xi-\xi_d}{L_2} + \EpsBeta b_{\rhoFR}(\xi)
  \]
  for all $\xi\in\cball{\delta}{\xi^*}$, so that
  \[
  \norm{\xi^*-\xi_d}{L_2} \leq \norm{\xi-\xi_d}{L_2} + \EpsBeta (b_{\rhoFR}(\xi)
  - b_{\rhoFR}(\xi^*) ).
  \]
  Exploiting the structure of $b_{\DelBeta, \rhoFR}$ and using the linearity of
  the integral operator we can write
  \[
  b_{\DelBeta, \rhoFR}(\xi) - b_{\DelBeta, \rhoFR}(\xi^*) = \int_0^T \sum_{j}
  \beta(-c_j(x(\tau), u(\tau))) - \beta(-c_j(x^*(\tau), u^*(\tau))) d\tau
  \]
  Using the fact that $\beta$ is a $\CC^2$ function and each $c_j$ is $\CC^2$ in
  both arguments (so that they are all bounded on $[0,T]$), there exists $c_2
  >0$ such that
  \[
  \begin{split}
    b_{\DelBeta, \rhoFR}(\xi) - b_{\DelBeta, \rhoFR}(\xi^*) &\leq T
    \max_{\tau\in[0,T]} \sum_{j} \beta(-c_j(x(\tau), u(\tau))) -
    \beta(-c_j(x^*(\tau), u^*(\tau))) \leq T c_2.
  \end{split}
  \]
  The result follows by choosing $\epsilon_2 = \EpsBeta T c_2$.

  To prove statement~(ii) we assume, without loss of generality, that the state
  vector can be written as $x = [x_1^T \; x_2^T]^T$, where $x_1\in\real^p$ is
  the vector of performance outputs, that is $p(x) = x_1$, and
  $x_2\in\real^{n-p}$ the remaining portion of the state. We use the same
  partition for any (state-input) curve so that, given a desired curve
  $\xi_d\in\Xtilde$, a weigthed $L_2$ distance of a curve $\xi\in\Xtilde$ from a
  lifted trajectory $\xi_d\in\TT$ satisfies
  \begin{equation}
    \begin{split}
      \norm{\xi-\xi_d}{L_2} \leq &\half \int_0^T \norm{x_1(\tau) -
        \subscr{x}{1d}(\tau)}{Q_1}^2 +
      \norm{x_2(\tau) - \subscr{x}{2d}(\tau)}{Q_2}^2 + \norm{u(\tau) -
        \subscr{u}{d}(\tau)}{R}^2 \; d\tau\\[1.2ex]
      & + \half \norm{x_1(T) - \subscr{x}{1d}(T)}{P_1}^2 + \half \norm{x_2(T) -
        \subscr{x}{2d}(T)}{P_2}^2 + c_3
    \end{split}
    \label{eq:L2_dist_x1x2}
  \end{equation}
  where $Q_1$, $Q_2$, $R$, $P_1$ and $P_2$ are positive definite matrices and
  $c_3>0$ is a positive constant taking into account cross terms.
  %
  %
  Now, rearranging terms in \eqref{eq:L2_dist_x1x2}, we have
  \begin{equation*}
    \begin{split}
      \norm{\xi-\xi_d}{L_2} \leq&\; \half \int_0^T \norm{x_1(\tau) -
        \subscr{x}{1d}(\tau)}{Q_1}^2 d\tau + \half \norm{x_1(T) -
        \subscr{x}{1d}(T)}{P_1}^2 +\\
      & \half \int_0^T \norm{x_2(\tau) - \subscr{x}{2d}(\tau)}{Q_2}^2 + \half
      \norm{x_2(T) - \subscr{x}{2d}(T)}{P_2}^2 + \half \int_0^T \norm{u(\tau) -
        \subscr{u}{d}(\tau)}{R}^2 \; d\tau + c_3\\[1.2ex]
      =&\; \norm{p(x(\cdot)) - \subscr{\yout}{d}(\cdot)}{L_2} + \norm{x_2(\cdot)
        - \subscr{x}{2d}(\cdot)}{L_2} + \norm{u(\cdot) -
        \subscr{u}{d}(\cdot)}{L_2} + c_3,
    \end{split}
  \end{equation*}
  where we have used the fact that $\subscr{\xi}{d}$ is a lifted trajectory (and
  thus $\subscr{x}{1d}(\cdot) := p(\subscr{x}{d}(\cdot)) =
  \subscr{\yout}{d}(\cdot)$) and supposed that the weighted $L_2$ norm for the
  input has no terminal penalty. 
  Using similar arguments, the converse inequality can be obtained, i.e.,
  $\norm{\xi-\xi_d}{L_2} \geq \norm{p(x(\cdot)) - \subscr{\yout}{d}(\cdot)}{L_2}
  + \norm{x_2(\cdot) - \subscr{x}{2d}(\cdot)}{L_2} + \norm{u(\cdot) -
    \subscr{u}{d}(\cdot)}{L_2} - c_4$,
  for a suitable $c_4>0$. It is worth noting that, if the weight matrices $Q$,
  $R$ and $P_f$ are block diagonal the above inequalities are equalities with
  $c_3 = c_4 = 0$.
	
  From (ii), we have that there exist $\xi^*\in\TT$ and $\epsilon_2>0$ such that
  $\norm{\xi^*-\xi_d}{L_2} \leq \norm{\xi-\xi_d}{L_2} + \epsilon_2$ for all
  $\xi\in\TT$ in a neighborhood of $\xi^*$. Thus, we can write
  \begin{equation*}
    \begin{split}
      \norm{p(x^*(\cdot)) - \subscr{\yout}{d}(\cdot)}{L_2} + \norm{x_2^*(\cdot)
        - \subscr{x}{2d}(\cdot)}{L_2} &+ \norm{u^*(\cdot) -
        \subscr{u}{d}(\cdot)}{L_2} \leq\\[1.2ex]
      &\norm{p(x(\cdot)) - \subscr{\yout}{d}(\cdot)}{L_2} + \norm{x_2(\cdot) -
        \subscr{x}{2d}(\cdot)}{L_2} + \norm{u(\cdot) -
        \subscr{u}{d}(\cdot)}{L_2} + \epsilon_2 + c_3 + c_4,
    \end{split}
  \end{equation*}
  and,
  \begin{equation*}
    \begin{split}
      \norm{p(x^*(\cdot)) - \subscr{\yout}{d}(\cdot)}{L_2} \leq & \;
      \norm{p(x(\cdot)) - \subscr{\yout}{d}(\cdot)}{L_2} + \\[1.2ex]
      & \!\!( \norm{x_2(\cdot) - \subscr{x}{2d}(\cdot)}{L_2} -
      \norm{x_2^*(\cdot) - \subscr{x}{2d}(\cdot)}{L_2} + \norm{u(\cdot) -
        \subscr{u}{d}(\cdot)}{L_2} - \norm{u^*(\cdot) -
        \subscr{u}{d}(\cdot)}{L_2} + \epsilon_2 + c_3 + c_4).
    \end{split}
  \end{equation*}
  Using the same arguments on boundedness of the state and input trajectories
  and boundedness of the weighted $L_2$ norm on a neighborhood of $\xi^*$ as in
  (ii), we have that there exists $\epsilon_3>0$ such that
  $\norm{p(x^*(\cdot)) - \subscr{\yout}{d}(\cdot)}{L_2} \leq \norm{p(x(\cdot)) -
    \subscr{\yout}{d}(\cdot)}{L_2} + \epsilon_3$,
  thus concluding the proof.
\end{proof}


\begin{remark}[Analysis of the lift and constraint strategy for the PVTOL]
  The initialization part of the lift and constrain strategy for the PVTOL can
  be easily performed. Indeed, the $\lift_0$ procedure simply implements
  equations in \eqref{eq:lift_pvtol0}.
  Then we proceed by analyzing each step. As for Step~1, we can use
  Theorem~\ref{thm:PVTOL_uncstr_lift} to prove that there exists $\EpsVTOLo>0$
  such that for any $\EpsVTOL<\EpsVTOLo$ an unconstrained lifted trajectory for
  the coupled PVTOL exists and depends continuously on the coupling
  parameter. Thus, we know that for ``small'' positive values of the coupling
  parameter the continuation procedure will be successful, thus providing an
  unconstrained desired trajectory for the following continuation steps. Using
  Theorem~\ref{thm:exist_feas_traj}~(i), we have that for feasibility regions
  that are not too tight Step~2 and Step~3 will be successful and, thus, a
  feasible trajectory that locally solves the practical feasible lifting task
  can be found.
\end{remark}

\section{Numerical computations on the PVTOL}
\label{sec:numerical_comp}
In this section we present numerical computations showing the effectiveness of
the proposed strategy on the PVTOL.
We proceed defining the feasibility region (equivalently the point-wise
constraints that the system trajectories must satisfy) and the desired maneuver.
Regarding the feasibility region, we consider the case of constraining
the input only, i.e. $u_1$ and $u_2$ are bounded with, in particular, $u_1$
strictly positive. In order to have a compact feasibility region we just assume
that the states must belong to a sufficiently large compact and simply connected
set $\Omega\subset\real^6$ such that the state portion of the trajectory is for
sure feasible. The feasibility set is thus defined as
\[
\cXU = \{(x,u)\in\Omega\times\real^2 | 0 < u_{1 \text{min}} \leq u_1 \leq u_{1
  \text{max}}, u_{2 \text{min}} \leq u_2 \leq u_{2 \text{max}} \; \text{for
  given} \; u_{1 \text{min}}, u_{1 \text{max}}, u_{2 \text{min}} \; \text{and}
\; u_{2 \text{max}}\}.
\]
Then, we parametrize the feasibility region with the scaling parameter $\rhoFR$
and define the inequalities determining the barrier functional. We pose the
inequalities in terms of the square distance of $u_1$ and $u_2$ from the
boundary values in order to have a smooth barrier functional. The two
inequalities are as follows
\[
\begin{split}
  \left[ u_1 - \frac{(u_{1 \text{max}} + u_{1 \text{min}})}{ 2 } \right]^2 &\leq
  \left[ (\rhoFR + (1-\rhoFR) k_1) \frac{(u_{1 \text{max}} - u_{1 \text{min}})
    }{ 2 }\right]^2\\
  \left[ u_2 - \frac{(u_{2 \text{max}} + u_{2 \text{min}})}{ 2 } \right]^2 &\leq
  \left[ (\rhoFR + (1-\rhoFR) k_2) \frac{(u_{2 \text{max}} - u_{2 \text{min}})
    }{ 2 }\right]^2
\end{split}
\]
where we set $u_{1 \text{min}} = 0.5 g$ ($g$ being the gravity constant), $u_{1
  \text{max}} = 1.5 g$ m/s$^2$, $u_{2 \text{min}} = -80$ deg/s$^2$, $u_{2
  \text{max}} = 80$ deg/s$^2$, while $k_1$ and $k_2$ are two scaling factors
that guarantee feasibility for $\rhoFR = 0$.

Regarding the desired maneuver, we aim at performing a barrel roll with a
constant velocity profile. Specifically, we choose the desired outputs,
$\subscr{y}{d}(\cdot)$ and $\subscr{z}{d}(\cdot)$, so that the desired path is
the one depicted in Figure~\ref{fig:path_unconstr} and the velocity is $v_d = 10
m/s$.

We are now ready to present the results of the main steps of the lift and
constrain strategy for the described feasibility region and desired maneuver. A
lifted trajectory for the decoupled PVTOL ($\EpsVTOL = 0$) is computed according
to Equations~\eqref{eq:lift_pvtol0}. Recall that this trajectory is a
quasi-static trajectory for any coupled PVTOL. Then, using the dynamic embedding
technique described in Section~\ref{sec:explor_strategy}, we compute a lifted
trajectory for the coupled PVTOL with nominal parameter $\EpsVTOL = 1$. We solve
the optimal control problem by using the projection operator Newton method.
The quasi-static versus lifted path is depicted in
Figure~\ref{fig:path_roll_rollrate_unconstr} together with the velocity and roll
trajectories.
\begin{figure}[htbp]
  \center \subfloat[path]{\label{fig:path_unconstr}
    \includegraphics[width = 0.32\textwidth]{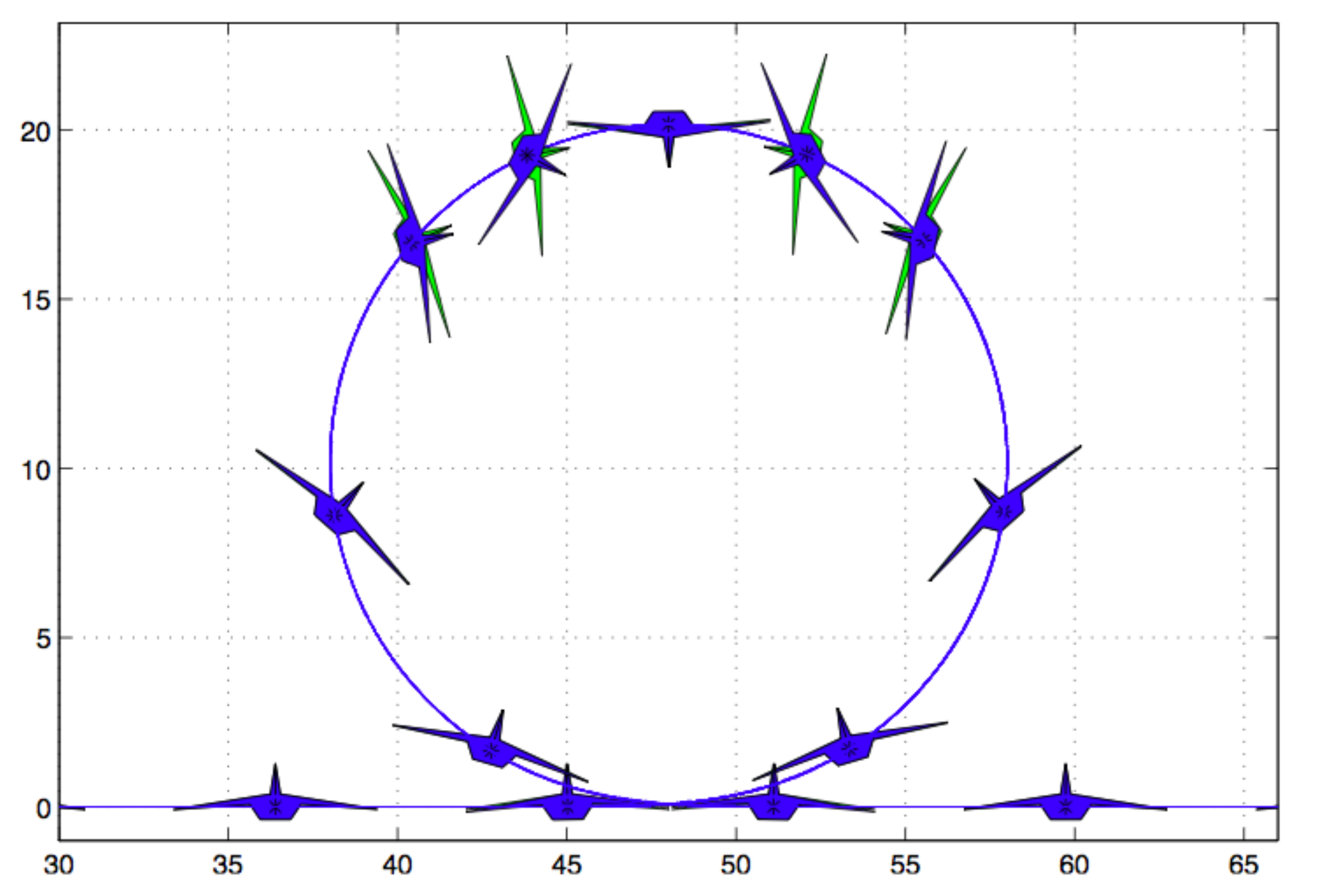}}\;
  \subfloat[velocity]{\label{fig:roll_unconstr}
    \includegraphics[width = 0.286\textwidth]{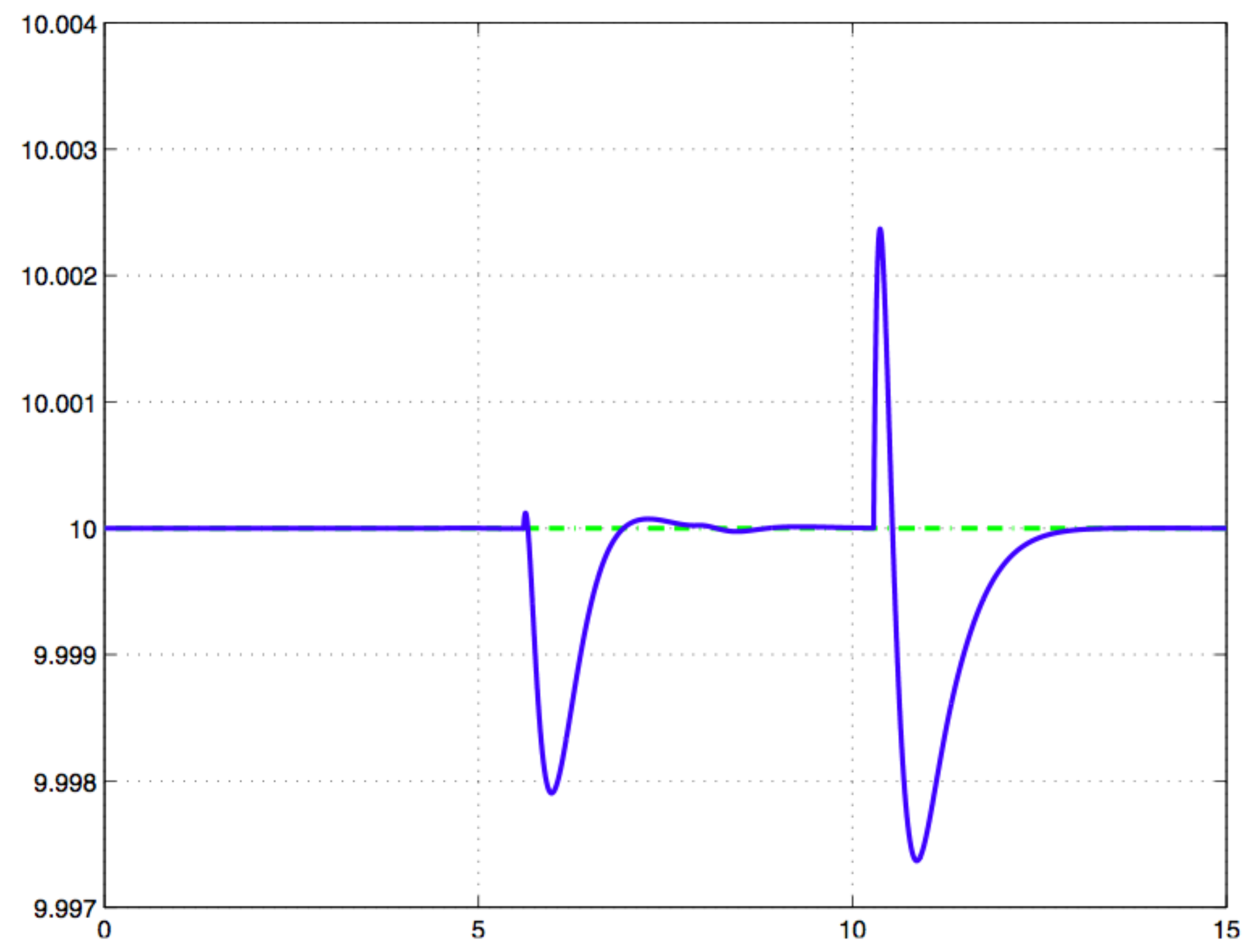}}\;
  \subfloat[roll]{\label{fig:roll-rate_unconstr}
    \includegraphics[width = 0.28\textwidth]{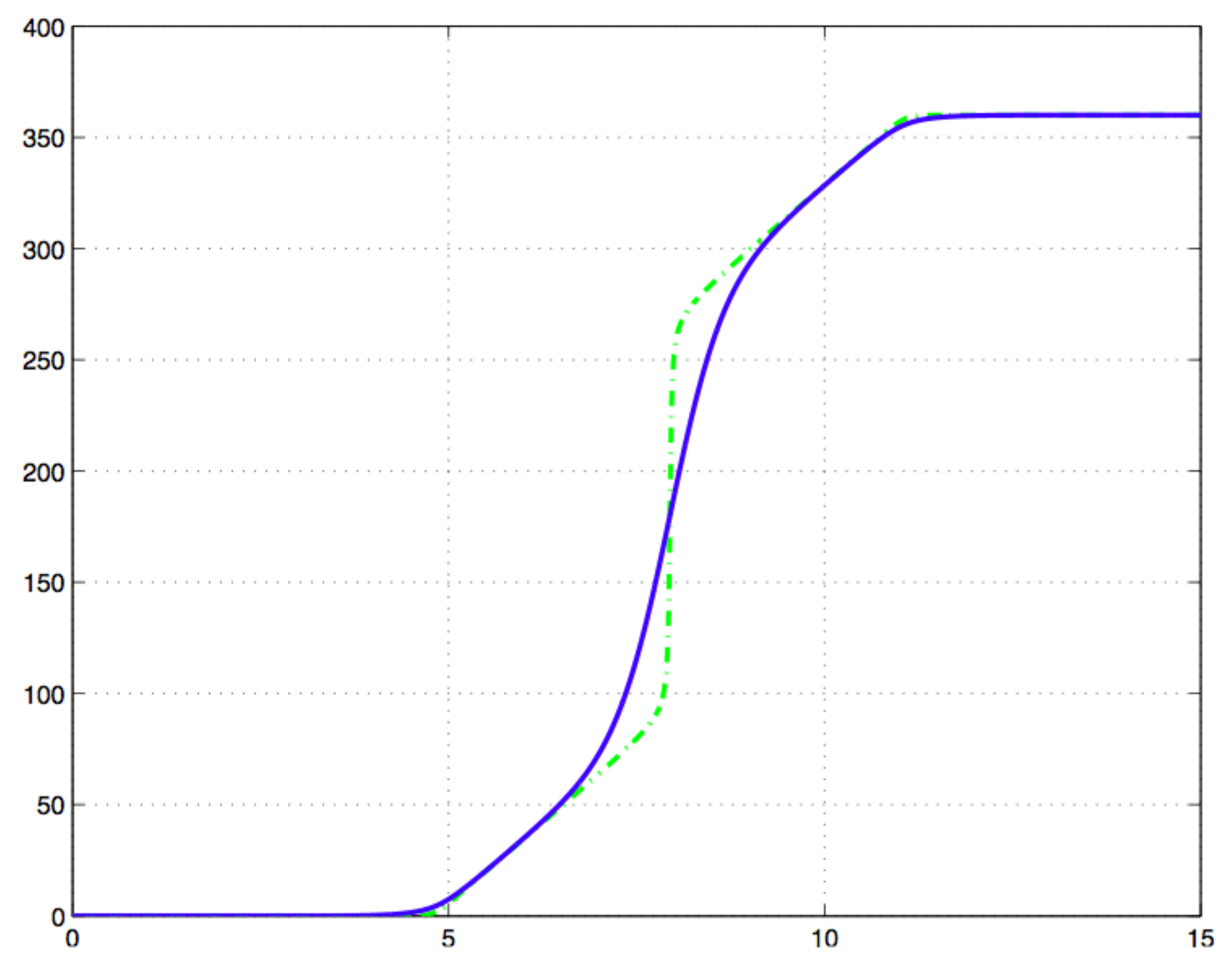}}
  \caption{Quasi-static (decoupled PVTOL) vs lifted (coupled PVTOL) path,
    velocity and roll trajectories for the unconstrained PVTOL ($\EpsVTOL =
    1$). Specifically: the dashed green lines are the quasi-static curves, while
    the solid blue are the lifted trajectories.}
  \label{fig:path_roll_rollrate_unconstr}
\end{figure}
The decoupled versus coupled roll-rate and input trajectories are depicted in
Figure~\ref{fig:u1_u2_unconstr}.
\begin{figure}[htbp]
  \center \subfloat[roll-rate]{\label{fig:roll_rate_unconstr}\includegraphics[width =
    0.298\textwidth]{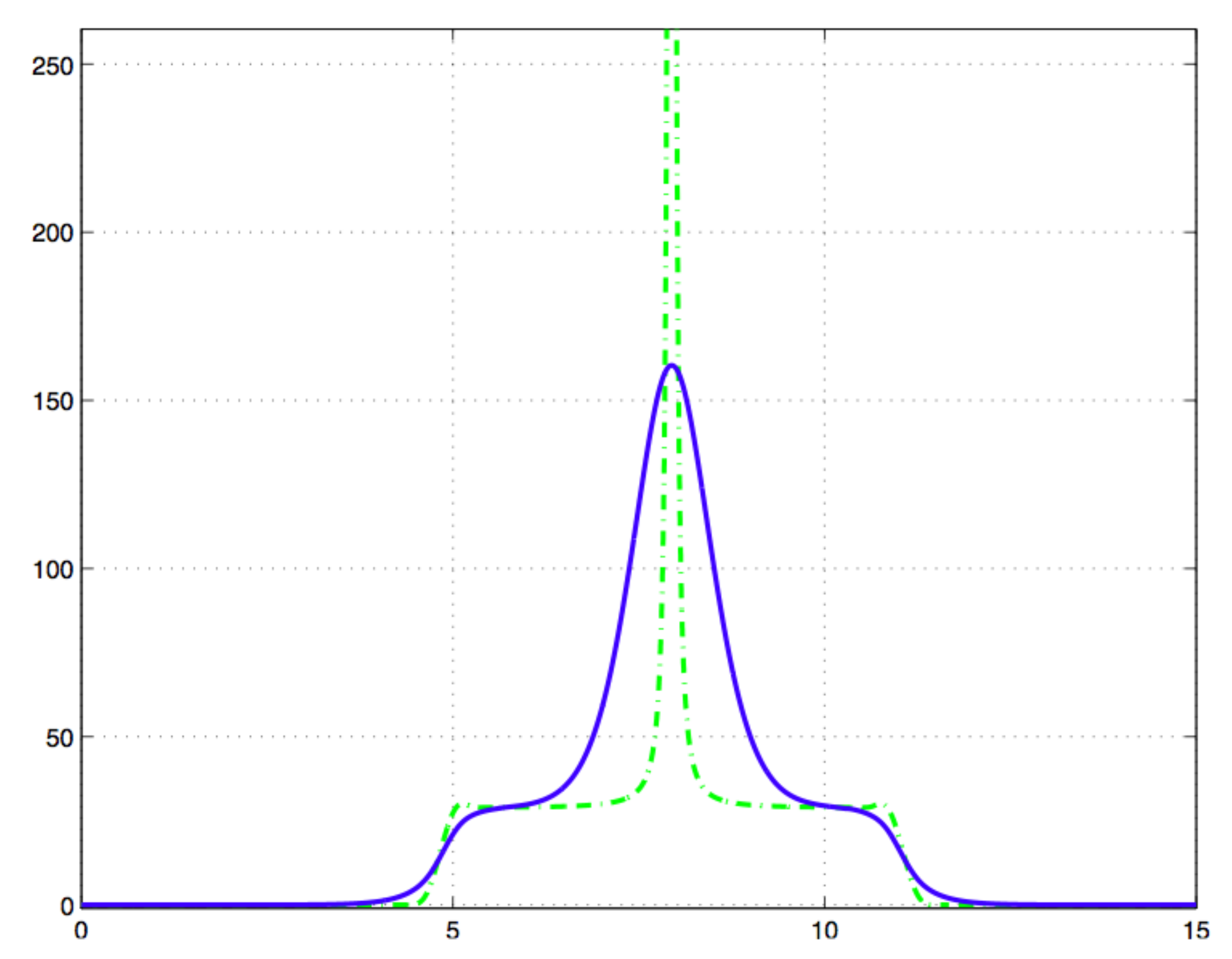}}\;
  \subfloat[$u_1$]{\label{fig:U1_unconstr}\includegraphics[width =
    0.296\textwidth]{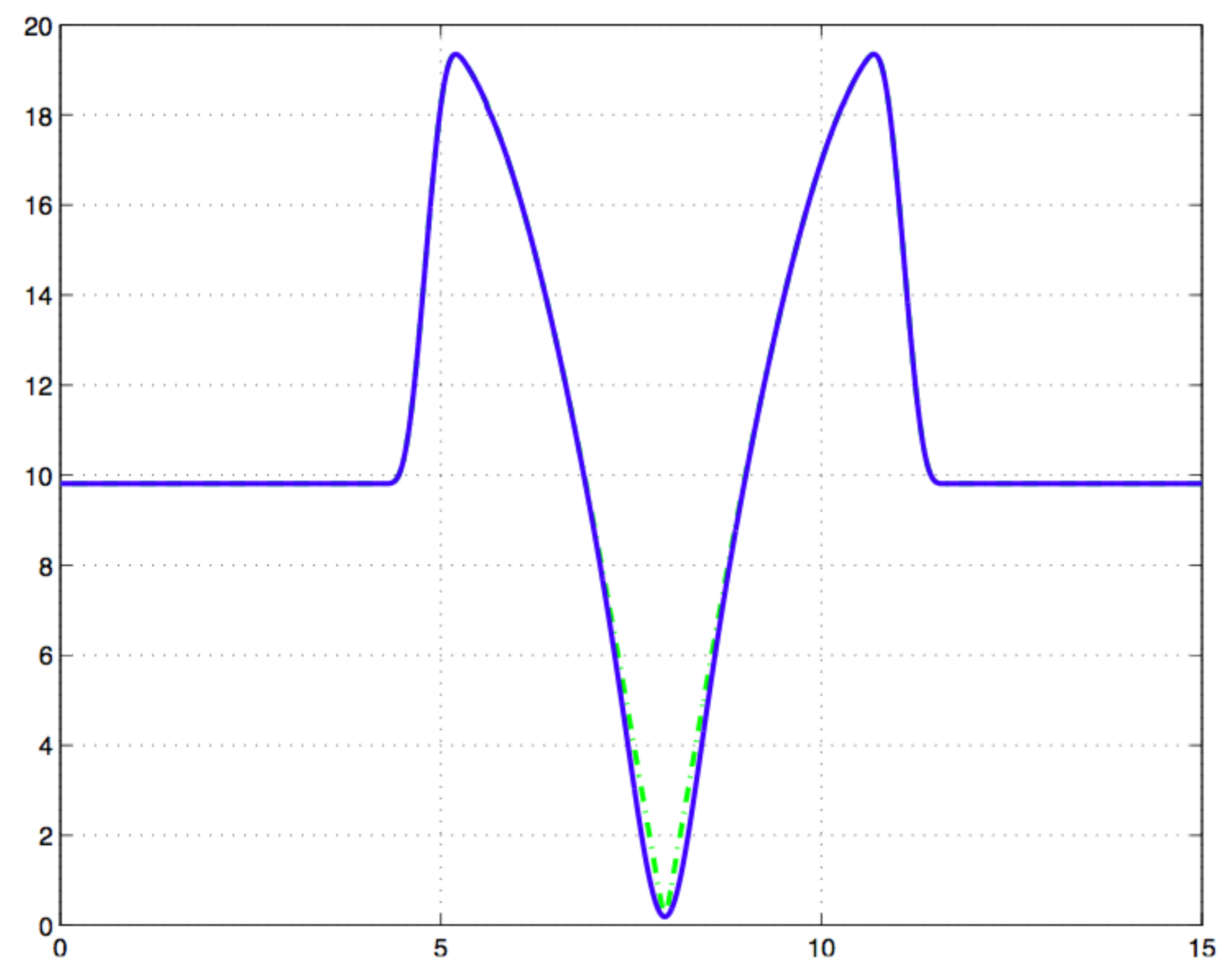}}\;
  \subfloat[$u_2$]{\label{fig:U2_unconstr}\includegraphics[width =
    0.305\textwidth]{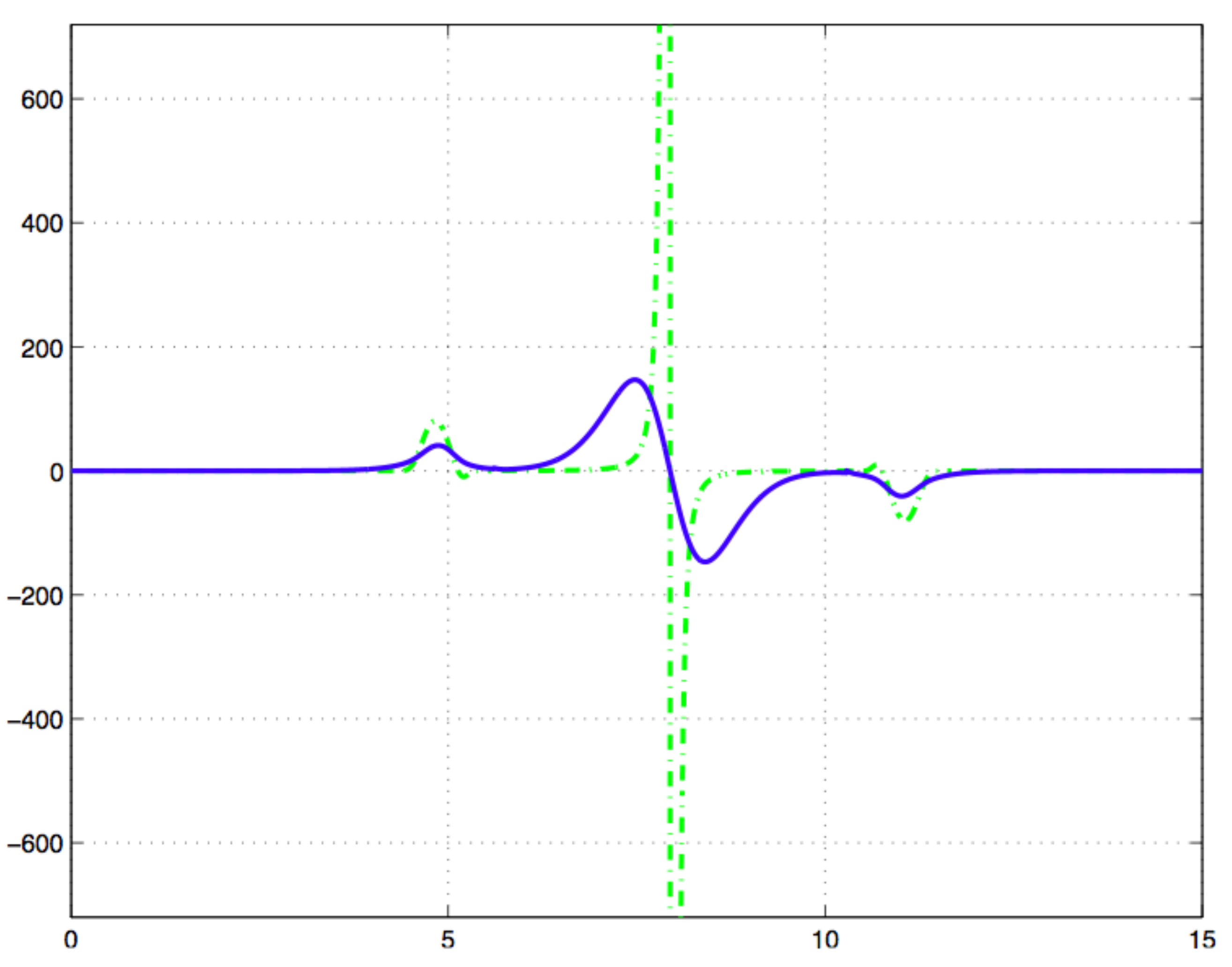}}\;
  \caption{Quasi-static (decoupled PVTOL) vs lifted (coupled PVTOL) roll rate and input
    trajectories for the unconstrained PVTOL ($\EpsVTOL = 1$). Specifically:
    the dashed green lines are the quasi-static curves, while the solid blue are
    the lifted trajectories. 
  }
  \label{fig:u1_u2_unconstr}
\end{figure}
As appears from the picture, the desired and lifted paths and velocities are
indistinguishable (the maximum error is of order $10^{-3}$ consistent with the
required absolute tolerance). This is consistent with the result in
Theorem~\ref{thm:PVTOL_uncstr_lift} stating that the lifting task can be solved
exactly for some positive values of $\EpsVTOL$. As expected the roll and roll
rate trajectories and the input trajectories are significantly different from
the quasi-static ones.
In particular, the optimal roll and roll rate trajectories display a degree of
anticipation and are smoother than the quasi-static. This is due to the
filtering action of the dynamics. This filtering effect can be seen also in the
snapshots of the PVTOL animation in Figure~\ref{fig:path_unconstr}.
It is worth noting that neither the desired nor the lifted trajectory satisfy
the constraints and, thus, are both infeasible.

Now, we are ready to show the results of the ``constrain'' part of the strategy.
As regards the $L_2$ weights, we choose diagonal $Q$ and $R$ matrices (for
simplicity) and penalize the outputs (external position states) $10^4$ times the
other states and the inputs.
We initialize the strategy by setting $\EpsBeta = 10$. We choose a relatively
high value for $\EpsBeta$ so that for each given value of $\rhoFR$ the
trajectory that we find is sufficiently far from the boundary. This has two
advantages. First, when we increase $\rhoFR$, thus shrinking the feasibility
region, the constraints are only slightly violated if the step-size on $\rhoFR$
is not too large. Second, once $\rhoFR = 1$ has been reached, we can converge to
a tighter approximation ($\EpsBeta \rightarrow 0$) in an interior point fashion.
The parameter $\rhoFR$ is varied with a step-size of $0.2$. For each value of
$\rhoFR$ the minimization takes few (less than $5$) Newton iterations (with an
absolute tolerance on the descent direction set to $10^{-6}$). Once
reached the value $\rhoFR=1$, $\EpsBeta$ is decreased down to $\EpsBeta = 0.1$.

The desired versus feasible path, velocity and roll trajectories are depicted in
Figure~\ref{fig:path_vel_roll} (for $\rhoFR = 1$ and $\EpsBeta =
0.1$). Snapshots of the PVTOL animation are shown Figure~\ref{fig:Path} as for
the unconstrained case. For the velocity and roll angle we also plot
intermediate non-optimal trajectories obtained during the continuation
procedure. In particular we plot the trajectories obtained for $\rhoFR = 0.6$
and $\rhoFR = 1$ with $\EpsBeta = 10$.
\begin{figure}[htbp]
  \center \subfloat[Path]{\label{fig:Path}
    \includegraphics[width = 0.31\textwidth]{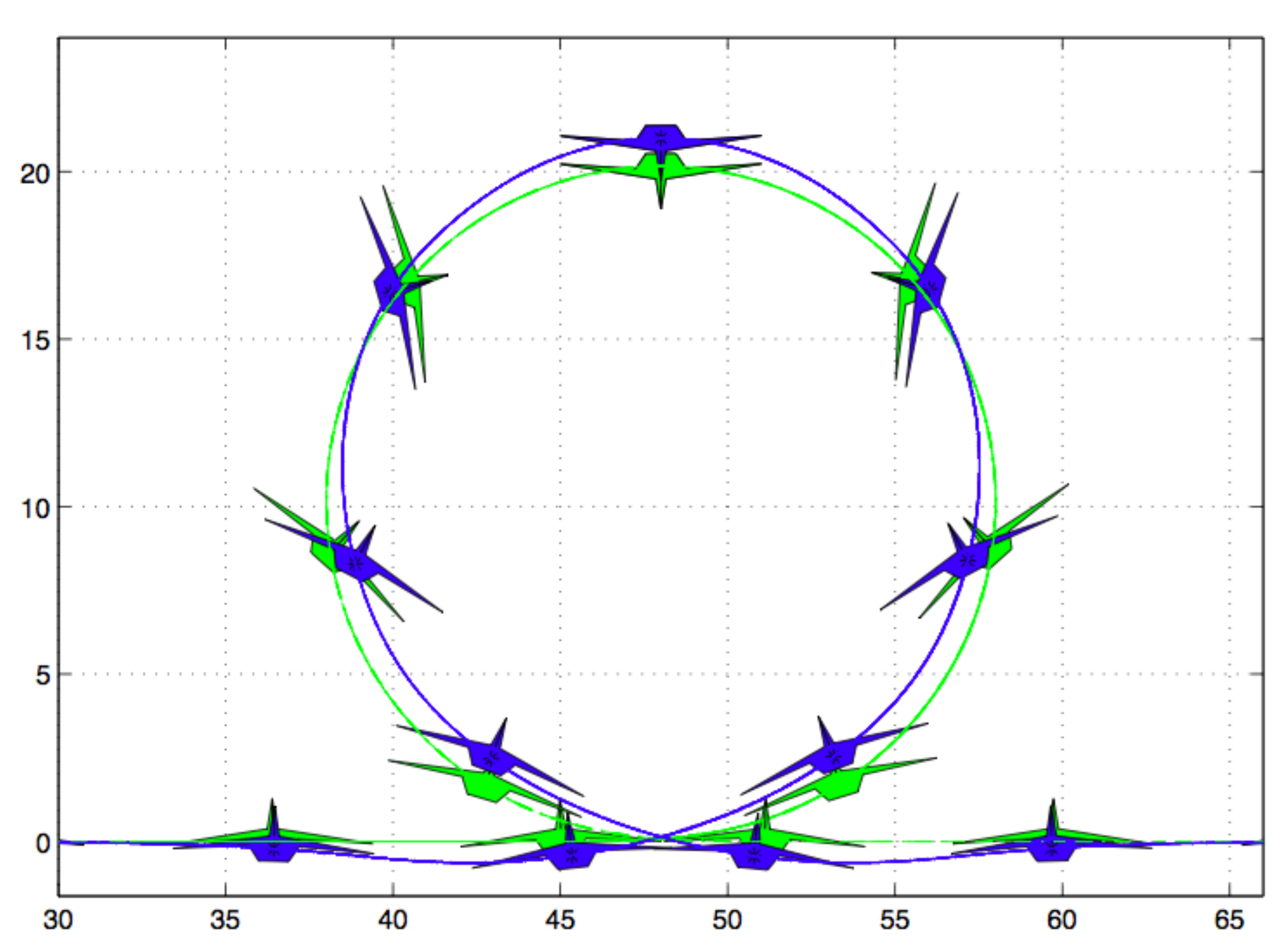}}\;
 \subfloat[Velocity]{\label{fig:Vel}
    \includegraphics[width = 0.29\textwidth]{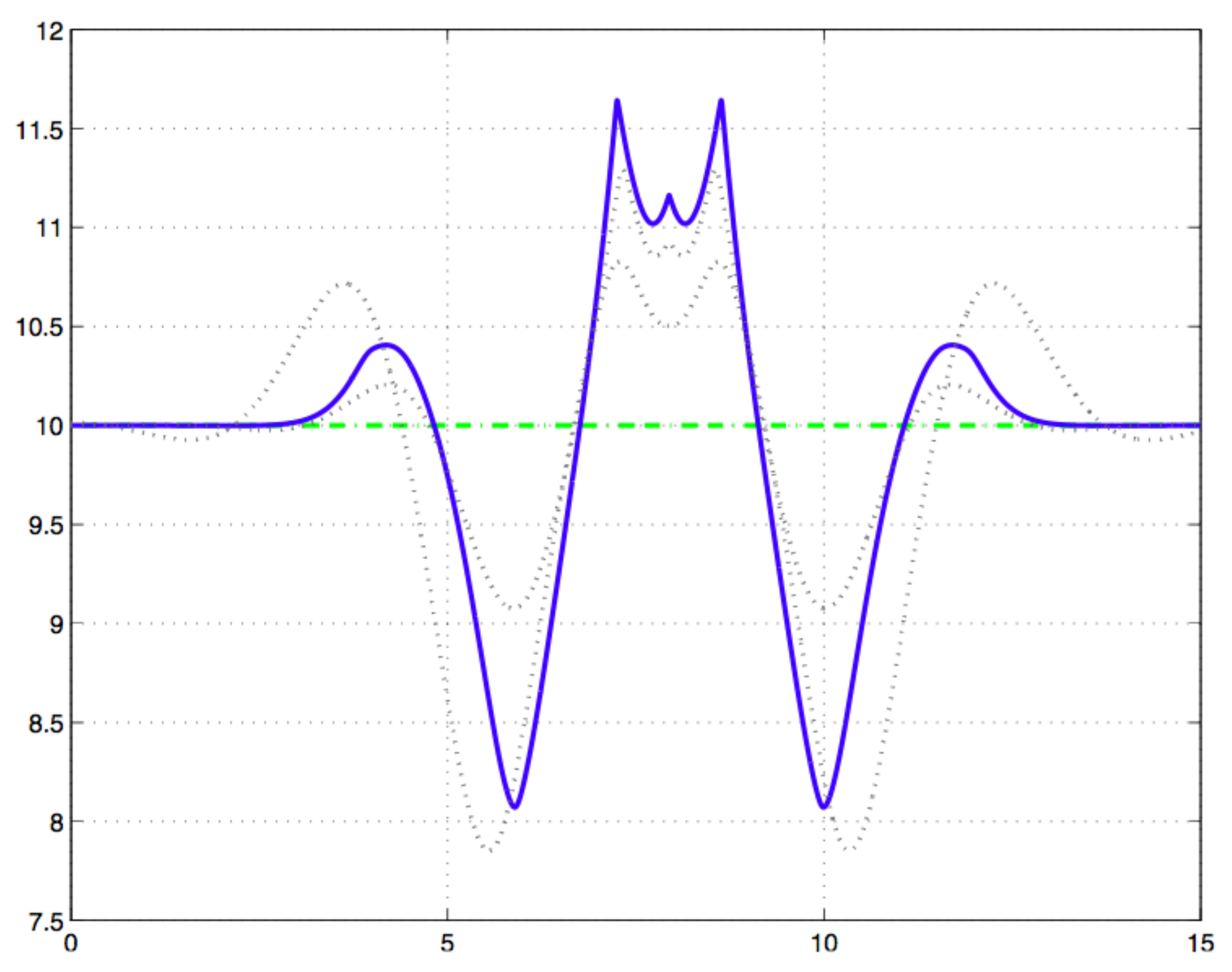}}\;
  \subfloat[Roll]{\label{fig:Roll}
    \includegraphics[width = 0.29\textwidth]{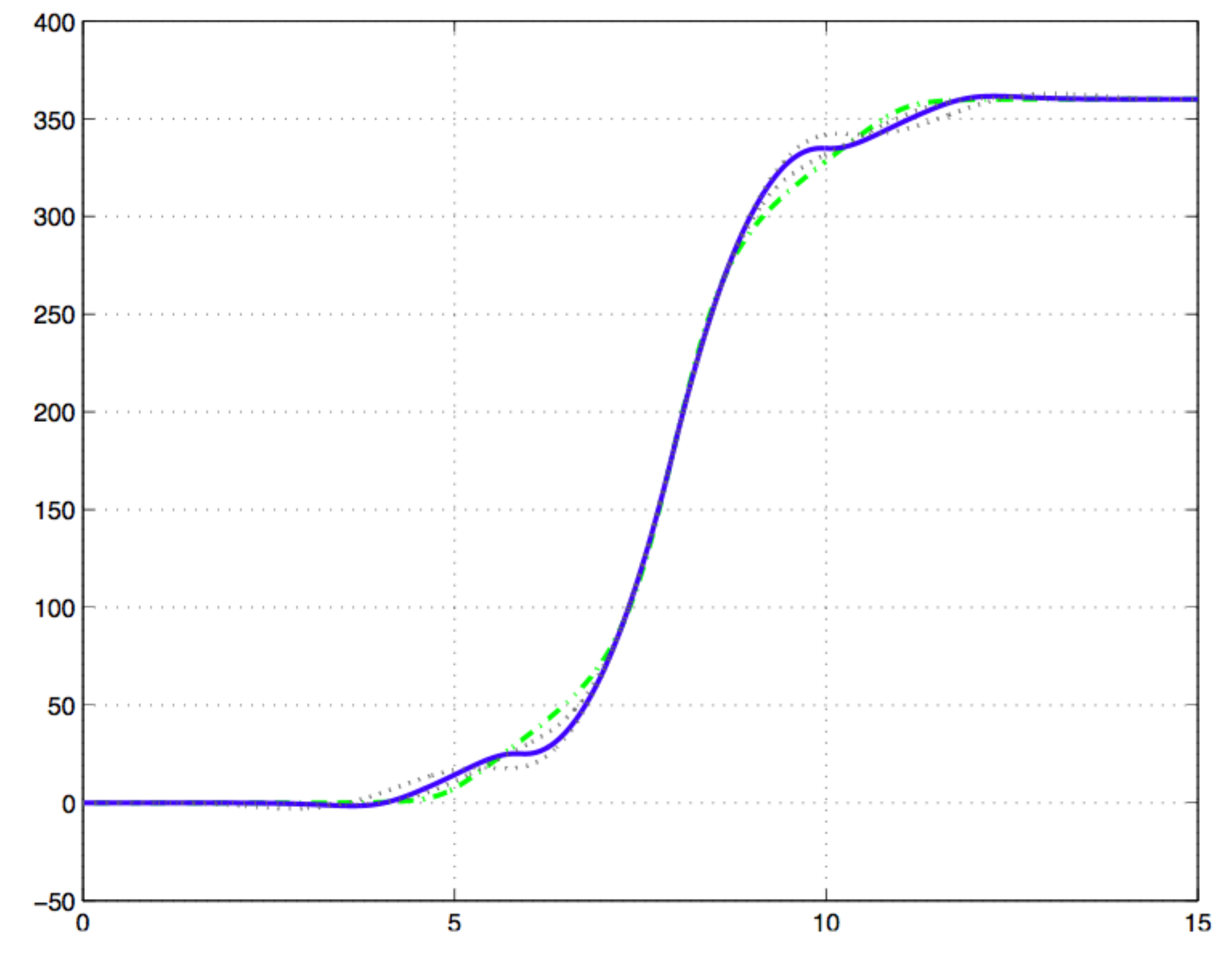}}
  \caption{Desired vs feasible path, velocity and roll angle for the coupled
    PVTOL ($\EpsVTOL=1$). Specifically: the dashed green lines are the desired
    infeasible curves, the solid blue lines are the feasible optimal
    trajectories ($\rhoFR = 1$ and $\EpsBeta = 0.1$), and (for velocity and roll
    angle) the thinner dotted grey lines are intermediate trajectories obtained
    for $\rhoFR = 0.6$ and $\rhoFR = 1$ with $\EpsBeta = 10$ respectively.}
  \label{fig:path_vel_roll}
\end{figure}
It is interesting to notice that even in presence of tight constraints on
the two inputs the feasible trajectory output (path and velocity) is reasonably
close to the desired one. The maximum error on $y(\cdot)$ and $z(\cdot)$ is
less than $1$m and the maximum error on the velocity is less than
$2$m/s. Also, it is worth noting that the roll trajectory,
Figure~\ref{fig:Roll}, stays bounded and relatively close to the desired one
even if the weight in the cost function is much lower than the one on the
positions.

In Figure~\ref{fig:roll-rate_u1_u2} we show the feasible versus desired
roll-rate and inputs for the same values of $\rhoFR$ and $\EpsBeta$. The
controls have a bang-bang like behavior. In particular they tend to assume the
boundary value in a larger interval than the one on which the constraints are
violated in order to compensate the missing control effort in the infeasible
time windows. This non-causal behavior shows that, in order to obtain
performances that are comparable with the unconstrained case the optimization
needs to work in a non-causal fashion.  
\vspace{-3ex}
\begin{figure}[htbp]
  \center \subfloat[roll-rate]{\label{fig:roll_rate_cstr}
    \includegraphics[width = 0.299\textwidth]{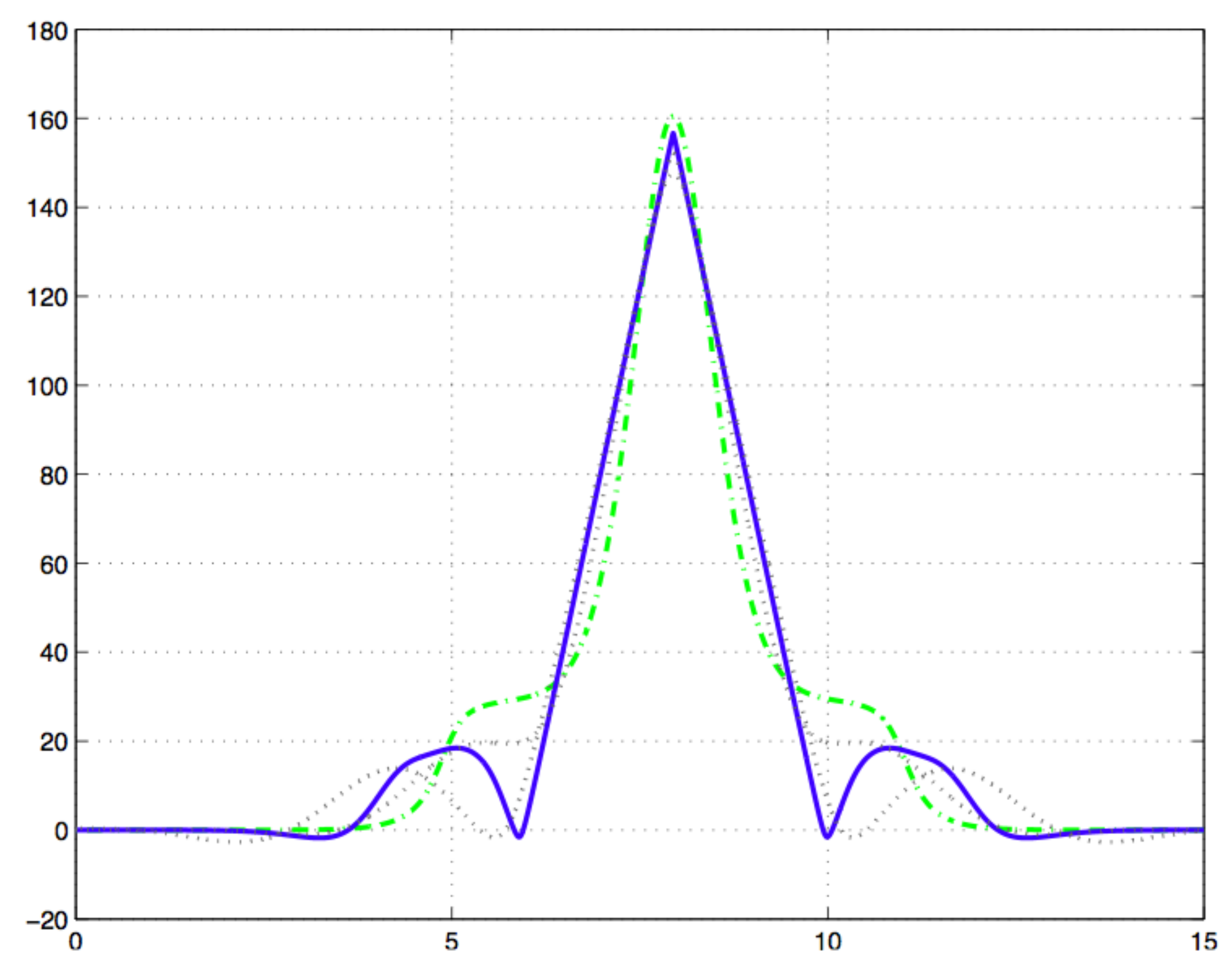}}\;
  \subfloat[$u_1$]{\label{fig:U1_cstr}
    \includegraphics[width = 0.299\textwidth]{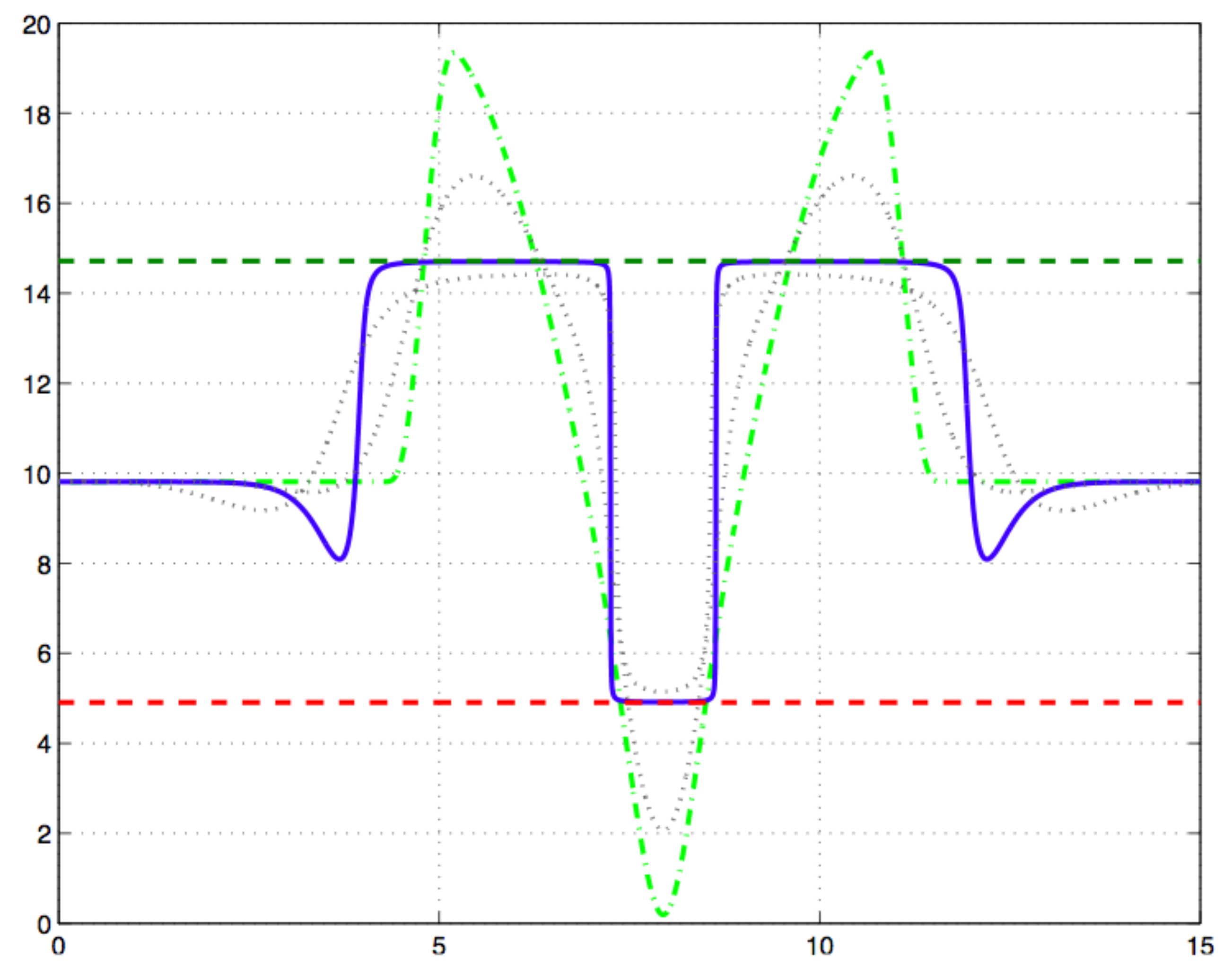}}\;
  \subfloat[$u_2$]{\label{fig:U2_zoom}
    \includegraphics[width = 0.299\textwidth]{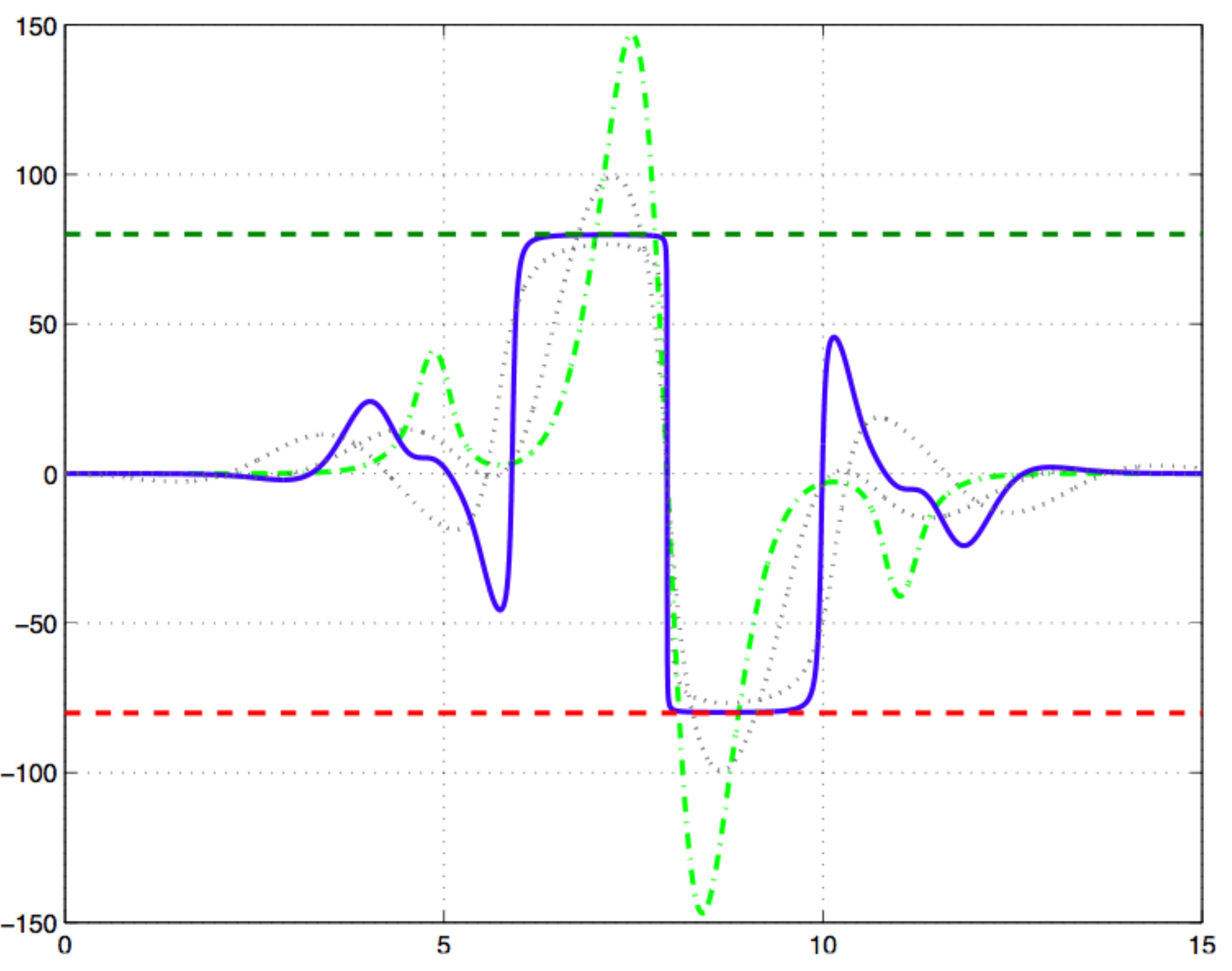}}\;
  \caption{Desired vs feasible roll-rate, $u_1$ and $u_2$ for
    the coupled PVTOL ($\EpsVTOL=1$). Specifically: the dashed green lines are
    the desired infeasible curves, the solid blue lines are the optimal
    trajectories ($\rhoFR = 1$ and $\EpsBeta = 0.1$), and the thinner dotted grey lines
    are intermediate trajectories(obtained for $\rhoFR = 0.6$ and $\rhoFR
    = 1$ with $\EpsBeta = 10$ respectively.}
  \label{fig:roll-rate_u1_u2}
\end{figure}


The results obtained in the above computations show that the inputs tend to be
discontinuous when $\EpsBeta \rightarrow 0$ thus suggesting that the constrained
minimizer is, in fact, discontinuous. In order to have a smoother input we could
vary the parameters of the strategy. We could, e.g., increase the parameter
$\EpsBeta$ thus obtaining a smoother trajectory (as the intermediate feasible
trajectory shown in Figure~\ref{fig:roll-rate_u1_u2} in dotted grey line) or
increase the $L_2$ input weights (the coefficients of the matrix $R$). Both the
two procedures increase the error on the desired output. Next, we show a
different procedure to obtain the same input regularization without loosing too
much in terms of output error. This procedure allows us to show that the
proposed strategy can easily deal with both state and input constraints without
any increase in the complexity of the strategy. We proceed by applying a dynamic
extension to the system, \cite{AI:95}. In the extended system the actual inputs are
two additional states, while their derivatives are the new inputs. 
\begin{equation*}
  \begin{split}
    \dot{x} &= f(x,w)\\
    \dot{w} &= v,
  \end{split}
\end{equation*} 
where $x = [y \,z \,\varphi \,\dot{y} \,\dot{z} \,\dot{\varphi}]$ and $w = [u_1
\, u_2]$ are the states of the extended system and $v = [\dot{u}_1 \,\dot{u}_2]$
is the input.

The results are shown in Figure~\ref{fig:trajs_extend}. We compare the feasible
path and original inputs ($u_1$ and $u_2$) obtained with and without dynamic
extension. In particular, for the extended system both the original inputs
(additional states) and the original input derivatives (new inputs) are
constrained.
\vspace{-3ex}
\begin{figure}[htbp]
  \center \subfloat[Path]{\label{fig:Path_extend}
    \includegraphics[width = 0.299999\textwidth]{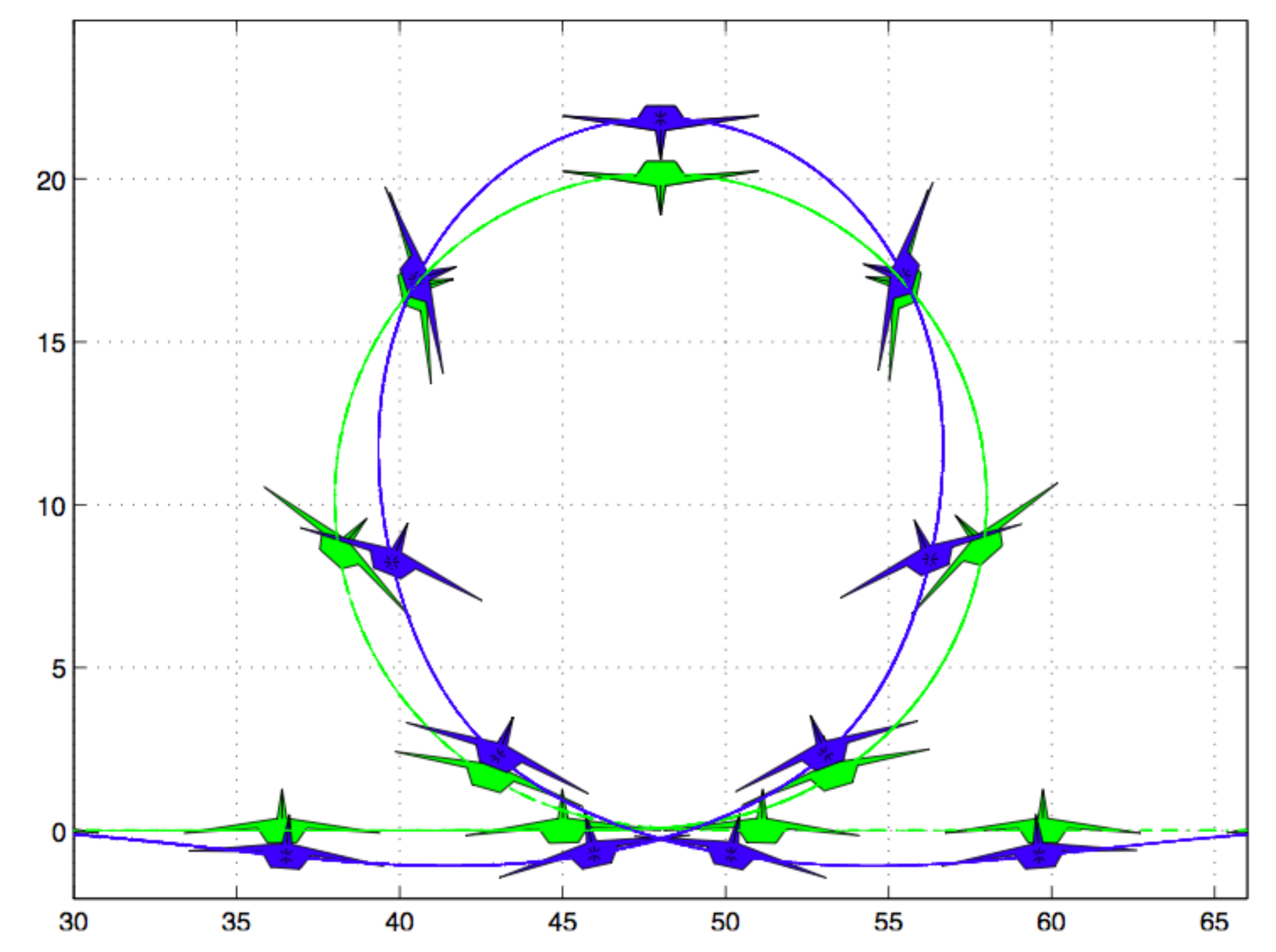}}\;
  \subfloat[$u_1$]{\label{fig:U1_extend}
    \includegraphics[width = 0.299999\textwidth]{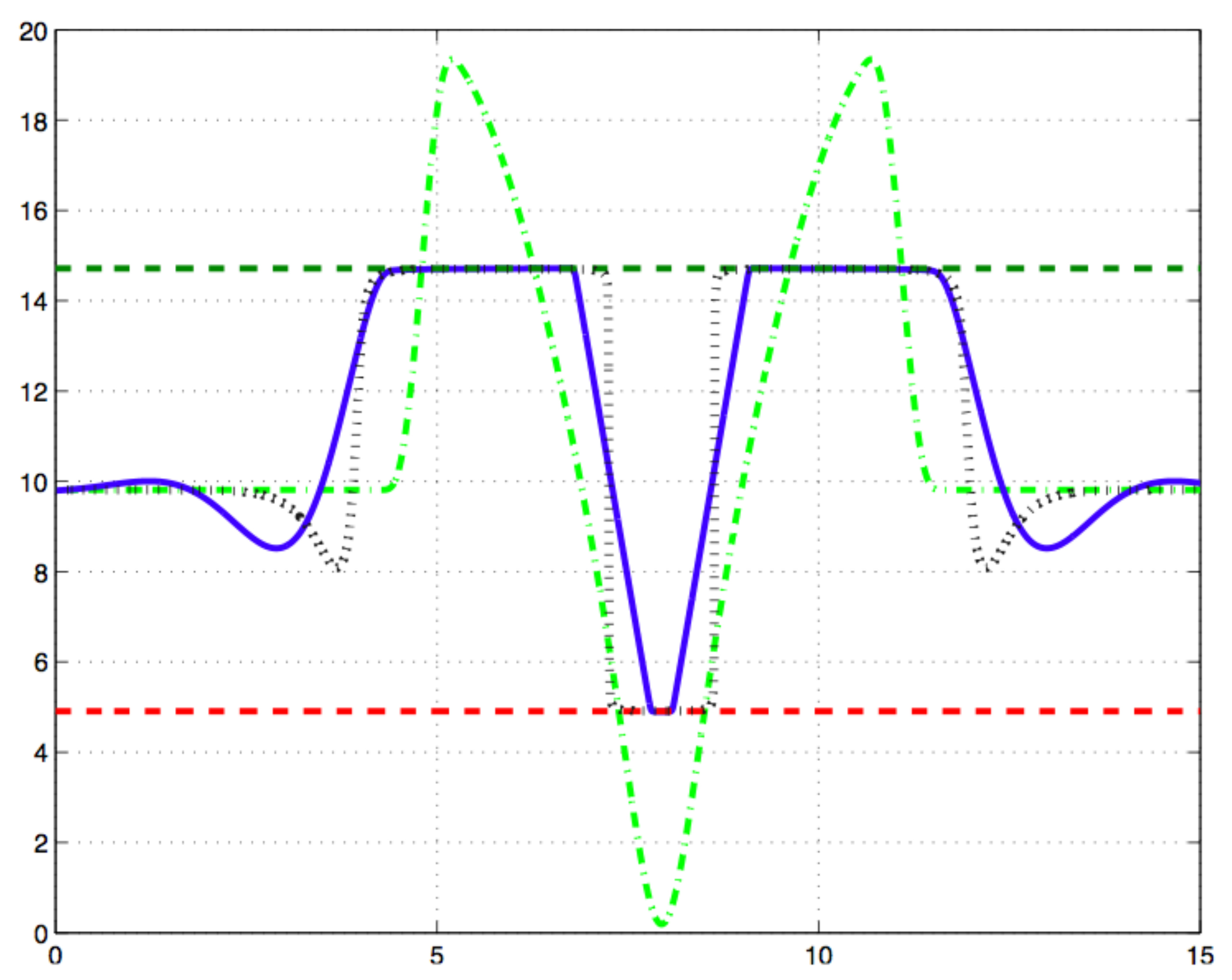}}\;
  \subfloat[$u_2$]{\label{fig:U2_extend}
    \includegraphics[width = 0.299999\textwidth]{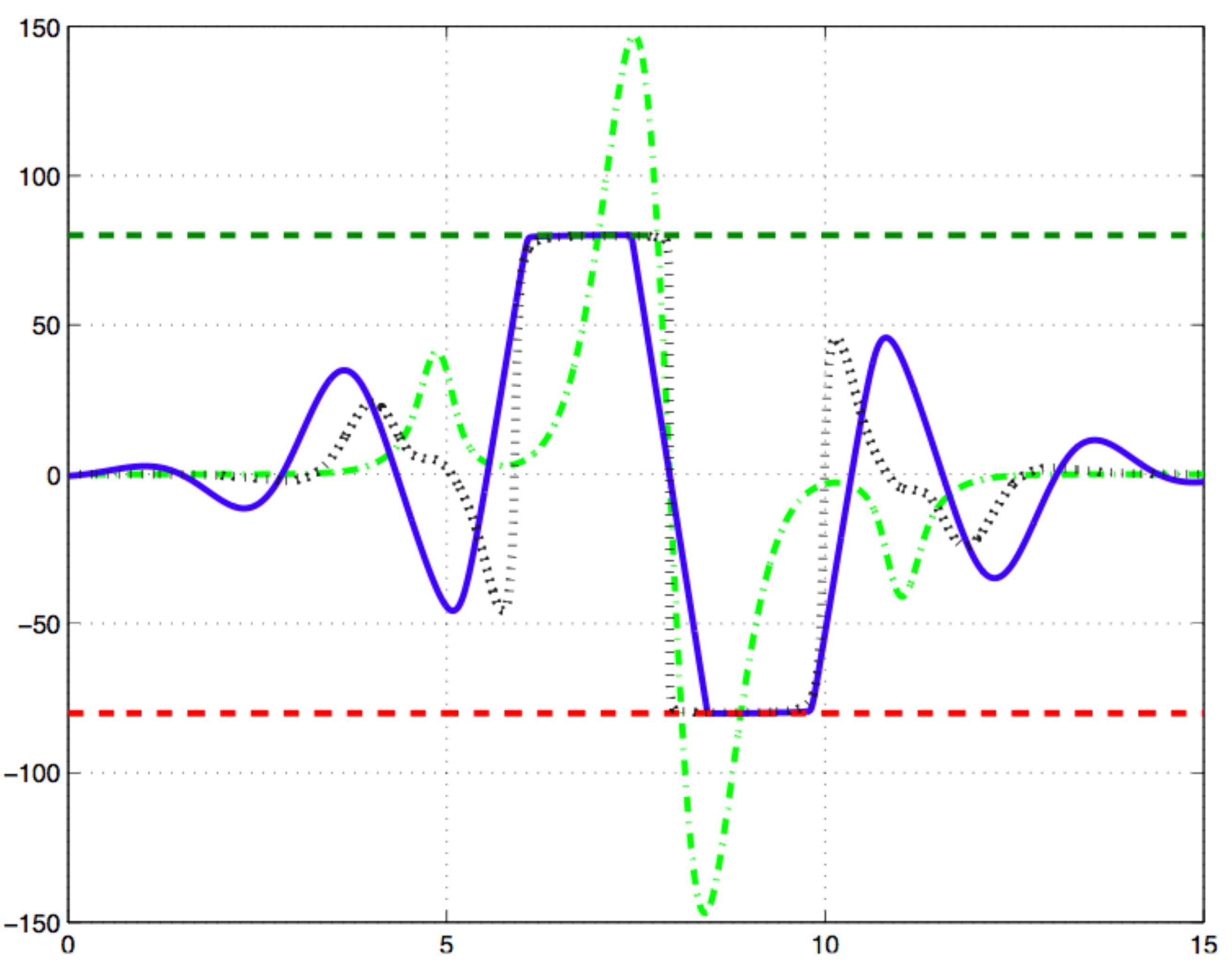}}
  \caption{Desired vs feasible path and inputs for the coupled PVTOL
    ($\EpsVTOL=1$) with and without bounds on the input derivatives. Specifically: the
    dashed green lines are the desired curves, the solid blue lines are the
    optimal trajectories for the system with dynamic extension (bounds on
    both inputs and input derivatives) and the dotted grey lines are the optimal
    trajectories for the system without dynamic extension (bounds on the
    inputs only).}
   \label{fig:trajs_extend}
\end{figure}


\section{Conclusions}
In this paper we have studied a constrained trajectory lifting problem for
nonlinear control systems. Given a desired output curve for a nonlinear
\emph{maneuvering} system, we compute a full (state-input) trajectory of the
system such that: (i) the output portion is close to the desired one, and (ii) a
set of point-wise state-input constraints are satisfied. We have proposed a
nonlinear optimal control strategy based on a novel projection operator based
Newton method for point-wise constrained control systems \cite{JH:02,JH-AS:06}
combined with dynamic embedding, constraints relaxation and continuation with
respect to parameters. Under suitable values of the system and constraints
parameters we have proven that a feasible trajectory exists and can be computed
by means of the proposed strategy. Finally, we have completely characterized the
strategy for the PVTOL aircraft and provided numerical computations showing the
effectiveness of the strategy for an aggressive desired barrel roll maneuver in
presence of relatively tight constraints.


\begin{thebibliography}{10}

\bibitem{SAA-NHM:02}
S.~A. Al-Hiddabi and N.~H. McClamroch.
\newblock Tracking and maneuver regulation control for nonlinear nonminimum
  phase systems: application to flight control.
\newblock {\em IEEE Transactions on Control Systems Technology}, 10(6), 2002.

\bibitem{SB-LV:04}
S.~Boyd and L.~Vandenberghe.
\newblock {\em Convex Optimization}.
\newblock Cambridge University Press, 2004.

\bibitem{CIB-AI:00}
A.~Isidori C.~I.~Byrnes.
\newblock Output regulation for nonlinear systems: an overview.
\newblock {\em International Journal of Robust and Nonlinear Control},
  10:323--337, 2000.

\bibitem{LC-MM-CN-MT:10}
L.~Consolini, M.~Maggiore, C.~Nielsen, and M.~Tosques.
\newblock Path following for the {PVTOL} aircraft.
\newblock {\em Automatica}, 46:1284--1296, 2010.

\bibitem{LC-MT:07}
L.~Consolini and M.~Tosques.
\newblock On the {VTOL} exact tracking with bounded internal dynamics via a
  poincar\'e map approach.
\newblock {\em IEEE Transactions on Automatic Control}, 52:1757--1762, 2007.

\bibitem{SD:99}
S.~Devasia.
\newblock Approximated stable inversion for nonlinear systems with
  nonhyperbolic internal dynamics.
\newblock {\em IEEE Transactions on Automatic Control}, 44(7):1419--1425, July
  1999.

\bibitem{SD-DC-BP:96}
S.~Devasia, D.~Chen, and B.~Paden.
\newblock Nonlinear inversion-based output tracking.
\newblock {\em IEEE Transactions on Automatic Control}, 41(7):930--942, July
  1996.

\bibitem{SD-BP:98}
S.~Devasia and B.~Paden.
\newblock Stable inversion for nonlinear nonminimum-phase time-varying systems.
\newblock {\em IEEE Transactions on Automatic Control}, 43(2):283--288,
  February 1998.

\bibitem{NHG-JEM:95}
N.~H. Getz and J.~E. Marsden.
\newblock Control for an autonomous bicycle.
\newblock In {\em {IEEE} Int. Conf. on Robotics and Automation}, volume~2,
  pages 1397--1402, May 1995.

\bibitem{JH:02}
J.~Hauser.
\newblock A projection operator approach to the optimization of trajectory
  functionals.
\newblock In {\em IFAC World Congress}, Barcelona, 2002.

\bibitem{JH-DGM:98}
J.~Hauser and D.~G. Meyer.
\newblock The trajectory manifold of a nonlinear control system.
\newblock In {\em {IEEE} Conf. on Decision and Control}, volume~1, pages
  1034--1039, December 1998.

\bibitem{JH-AS:06}
J.~Hauser and A.~Saccon.
\newblock A barrier function method for the optimization of trajectory
  functionals with constraints.
\newblock In {\em {IEEE} Conf. on Decision and Control}, pages 864--869, San
  Diego, Dec 2006.

\bibitem{JH-AS-RF:04}
J.~Hauser, A.~Saccon, and R.~Frezza.
\newblock Aggressive motorcycle trajectories.
\newblock In {\em {IFAC} Symposium on Nonlinear Control Systems}, Stuttgart,
  2004.

\bibitem{JH-AS-RF:05}
J.~Hauser, A.~Saccon, and R.~Frezza.
\newblock On the driven inverted pendulum.
\newblock In {\em {IEEE} Conf. on Decision and Control and European Control
  Conference}, pages 6176--6180, Dec. 2005.

\bibitem{JH-SSS-GM:92}
J.~Hauser, S.~Sastry, and G.~Meyer.
\newblock Nonlinear control design for slightly nonminimum phase systems:
  Application to {V/STOL} aircraft.
\newblock {\em Automatica}, 28(4):665--679, 1992.

\bibitem{JH:04}
J.~Huang.
\newblock {\em Nonlinear output regulation. Theory and applications}.
\newblock SIAM, Philadelphia, 2004.

\bibitem{LRH-GM:97}
L.R. Hunt and G.~Meyer.
\newblock Stable inversion for nonlinear-systems.
\newblock {\em Automatica}, 33(8):1549--1554, 1997.

\bibitem{AI:95}
A.~Isidori.
\newblock {\em Nonlinear Control Systems}.
\newblock Communications and Control Engineering Series. Springer, 3 edition,
  1995.

\bibitem{AI-CIB:90}
A.~Isidori and C.I. Byrnes.
\newblock Output regulation of nonlinear systems.
\newblock {\em IEEE Transactions on Automatic Control}, 35(2):131--140,
  February 1990.

\bibitem{AI-LM-AS:03}
A.~Isidori, L.~Marconi, and A.~Serrani.
\newblock {\em Robust Autonomous Guidance. An Internal Model Approach}.
\newblock Advances in Industrial Control. Springer, 2003.

\bibitem{LM-AI-AS:02}
L.~Marconi, A.~Isidori, and A.~Serrani.
\newblock Autonomous vertical landing on an oscillating platform: an internal
  model based approach.
\newblock {\em Automatica}, 38:21--32, 2002.

\bibitem{HM-HJP:94}
H.~Maurer and H.~J. Pesch.
\newblock Solution differentiability for nonlinear parametric control problems.
\newblock {\em SIAM Journal on Control and Optimization}, 32(6):1542--1554,
  1994.

\bibitem{HM-HJP:95}
H.~Maurer and H.~J. Pesch.
\newblock Solution differentiability for parametric nonlinear control problems
  with control-state constraints.
\newblock {\em Journal of Optimization Theory and Applications},
  86(2):285--309, 1995.

\bibitem{RN-LM:11}
R.~Naldi and L.~Marconi.
\newblock Optimal transition maneuvers for a class of {V/STOL} aircraft.
\newblock {\em Automatica}, 47, 2011.

\bibitem{GN-JH-RF:07}
G.~Notarstefano, J.~Hauser, and R.~Frezza.
\newblock Computing feasible trajectories for control-constrained systems: the
  {PVTOL} aircraft.
\newblock In {\em {IFAC} Symposium on Nonlinear Control Systems}, Pretoria, SA,
  August 2007.

\bibitem{GN-JH-RF:05}
G.~Notarstefano, J.~Hauser, and R.Frezza.
\newblock Trajectory manifold exploration for the {PVTOL} aircraft.
\newblock In {\em {IEEE} Conf. on Decision and Control and European Control
  Conference}, pages 5848--5853, Seville, December 2005.

\bibitem{AP-NVW-HN:07}
A.~Pavlov, N.~Van de~Wouw, and H.~Nijmeijer.
\newblock Global nonlinear output regulation: Convergence-based controller
  design.
\newblock {\em Automatica}, 43:456--463, 2007.

\bibitem{AP-KYP:08}
A.~Pavlov and K.~Y. Pettersen.
\newblock A new perspective on stable inversion of non-minimum phase nonlinear
  systems.
\newblock {\em Modeling, Identification and Control}, 29(1):29--35, 2008.

\bibitem{PM-SD-BP:96}
P.Martin, S.~Devasia, and B.~Paden.
\newblock A different look at output tracking: control of a {VTOL} aircraft.
\newblock {\em Automatica}, 32(1):101--107, 1996.

\bibitem{MWS-DJB:95}
M.~W. Spong and D.J. Block.
\newblock The pendubot: a mechatronic system for control research and
  education.
\newblock In {\em {IEEE} Conf. on Decision and Control}, pages 555--556,
  December 1995.

\bibitem{EZ:95}
Eberhard Zeidler.
\newblock {\em Applied Functional Analysis: Main Principles and their
  applications}.
\newblock Springer-Verlag, New York, 1995.

\end{thebibliography}

\appendix


\section{The \emph{Projection Operator} approach for the optimization of
  trajectory functionals}
\label{sec:prelim_proj_oper}
In this section we recall the main mathematical tools that we use to develop the
feasible trajectory exploration strategy and to prove its correctness.

\subsection{Trajectory tracking projection operator}
The trajectory tracking projection operator, \cite{JH-DGM:98}, provides a
numerically robust representation of nonlinear system trajectories and is at the
basis of the novel descent methods for nonlinear optimization of trajectory
functionals, \cite{JH:02}, used in the paper.
Let us consider the nonlinear control system $\dot x(t) = f(x(t), u(t))$, $x(0)
= x_0$, where $f(x,u)$ is a $\CC^2$ map in $x\in\real^n$ and $u\in\real^m$.
We recall that a \emph{bounded} curve $\eta = (\bar{x}(\cdot), \bar{u}(\cdot))$
is a (state-input) \emph{trajectory} of the system if $\dot{\bar{x}}(t) =
f(\bar{x}(t), \bar{u}(t))$, $\bar{x}(0) = x_0$, for all $t \in [0, T]$, $0<T\leq
+\infty$.
%
Suppose that $\xi(t) = (\alpha(t), \mu(t))$, $t
\in [0, T]$, is a bounded curve (e.g., an approximate
trajectory of the system) and let $\eta(t) = (x(t), u(t))$, $t \in [0, T]$, be
the trajectory determined by the nonlinear feedback system
\begin{equation*}
  \begin{split}
    \dot x(t) =& f(x(t), u(t)), \qquad x(0) = x_0,\\
    u(t) =& \mu(t) + K(t)(\alpha(t) - x(t))
  \end{split}
\end{equation*}
Under certain conditions on $f$ and $K$, this feedback system defines a
continuous, nonlinear \emph{projection operator}
\begin{equation*}
  \PP : = (\alpha, \mu) \mapsto \eta = (x, u).
\end{equation*}
That is, independent of $K$, if $\xi$ is a trajectory, $\xi\in\TT$, then $\xi$
is a fixed point of $\PP$, $\xi = \PP(\xi)$. If $K$ is bounded
(and, if $\xi$ is a trajectory of infinite extent, such that the above feedback
exponentially stabilizes $\xi_0$), then $\PP$ is well defined on an $L_\infty$
neighborhood of $\xi_0$ and is $\CC^r$ (with respect to the $L_\infty$ norm) on
its domain (including an open neighborhood of $\xi_0$) whenever $f$ is
\cite{JH:02}.
%
%
The first derivative of the projection operator, $\zeta \mapsto
D\PP(\xi)\cdot\zeta$, is the (continuous) linear projection operator given by
the standard linearization
\begin{equation*}
  \begin{split}
    \dot{z}(t) & = A(\eta(t)) z(t) + B(\eta(t)) v(t), \qquad z(0) = 0, \\
    v(t) & = \nu(t) + K(t) [\beta(t) - z(t) ] \; .
  \end{split}
  \label{eq:linz}
\end{equation*}
where $D\PP(\xi)\cdot\zeta = (z(\cdot), v(\cdot))$, with $\zeta = (\beta(\cdot),
\nu(\cdot))$, and $A(\eta(t)) = f_x(x(t), u(t))$ and $B(\eta(t)) = f_u(x(t),
u(t))$.  The tangent space $T_\xi\TT$ at a given trajectory $\xi\in\TT$ is,
thus, the set of curves $\zeta$ satisfying $\zeta = D\PP(\xi) \cdot \zeta$.

The projection operator $\PP$ provides a convenient parametrization of the
trajectories in the neighborhood of a given trajectory \cite{JH-DGM:98}. Indeed,
the tangent space $T_\xi\TT$ can be used to parameterize all nearby
trajectories.

\begin{theorem}[Trajectory manifold representation theorem \cite{JH-DGM:98}]
  \label{thm:traj_man_repr}
  Given $\xi \in \TT$, there is an $\epsilon > 0$ such that, for each $\eta \in
  \TT$ with $\norm{\eta - \xi}{L_\infty} < \epsilon$ there is a unique $\zeta \in T_\xi
  \TT$ such that $\eta = \PP(\xi+\zeta)$. This provides a $\CC^r$ atlas of
  charts, indexed by trajectories $\xi\in\TT$, so that $\TT$ is a $\CC^r$ Banach
  manifold.
\end{theorem}

\subsection{Projection operator based Newton method}%
\label{subsec:proj-operator}
Consider the unconstrained optimal control problem
\begin{equation}
  \begin{split}
    \minimize  & \; \int_0^T l(\tau, x(\tau), u(\tau)) d\tau + m(x(T))\\[1.2ex]
    \subj & \; \dot{x}(t) = f(x(t), u(t)), \qquad x(0)=x_0.
  \end{split}
  \label{eq:unconstr_opt_contr}
\end{equation}
where $l(t, x, u)$ is $\CC^2$ in $x$ and $u$, convex in $u$, and $\CC^1$ in $t$,
and $m(x)$ is $\CC^2$ in $x$.
%
This problem is equivalent to the constrained optimization problems
\begin{equation*}
  \min_{\xi \in \TT} h(\xi) = \min_{ \xi=\PP(\xi)} h(\xi).
\end{equation*}
where $h(\xi) := \int_0^T l(\tau, x(\tau), u(\tau)) d\tau + m(x(T))$ and the
constraint set, the trajectory space $\TT$, is a Banach submanifold of $\Xtilde
= (x_0, 0) + \Xinf$.
%
%
Next lemma is the basis for the Projection Operator Newton descent method that
we use in our strategy. This result is also useful to prove the existence of a
feasible lifted trajectory.

\begin{lemma}[Unconstrained minimization through projection \cite{JH:02}]
  \label{lem:min_via_projection}
  Let
  $g(\xi) := h(\PP(\xi))$,
  for $\xi \in \UU \subset \Xtilde$ with $\PP(\UU) \subset \UU \subset \dom{}
  \PP$. Then, the optimization problems
  \begin{equation*}
    \min_{\xi \in \TT} h(\xi) \quad \text{and} \quad \min_{\xi \in \UU} g(\xi)
    \label{eq:optim_contr_equiv_PO}
  \end{equation*}
  are equivalent in the following sense. If $\xi^* \in \TT \cap \UU$ is a
  constrained local minimum of $h$, then it is an unconstrained local minimum of
  $g$. If $\xi^+ \in \UU$ is an unconstrained local minimum of $g$ in $\UU$,
  then $\xi^* = \PP(\xi^+)$ is a constrained local minimum of $h$ on $\TT$.
\end{lemma}

The projection operator based Newton method, \cite{JH:02}, is the following.
\begin{center}
  \begin{minipage}{0.5\linewidth}
    \medskip \hrule width \linewidth \smallskip
    {\bf Algorithm} (projection operator Newton method)\\
    Given initial trajectory $\xi_0 \in \TT$\\
    {\bf For} $i = 0, 1, 2 ...$ %
    \begin{itemize}
    \item[] design $K$ defining $\PP$ about $\xi_i$
    \item[] search direction:
      $$\zeta_i = \text{arg} \min_{\zeta\in T_{\xi_i} \TT} Dg(\xi_i) \cdot \zeta + \frac{1}{2} D^2 g(\xi_i)(\zeta,
      \zeta)$$
    \item[] step size: $\gamma_i = \arg \min_{\gamma \in (0,1]} g(\xi +
      \gamma \zeta_i)$;
    \item[] project: $\xi_{i+1}={\PP}(\xi_i + \gamma_i \zeta_i)$.
    \end{itemize}
    {\bf end}
    \medskip \hrule width \linewidth \smallskip
  \end{minipage}
\end{center}
It is worth noting that the two main steps of designing the $K$ and searching
for the descent direction involve the solution of suitable (well known) LQ
optimal control problems.

\subsection{Barrier functional approach for constrained optimal control}
\label{sec:barrier_functional}
Next, we present an interior point method, introduced in \cite{JH-AS:06}, for
the optimization of trajectory functionals in presence of point-wise state and
input constraints. The objective is to solve over the class of bounded inputs
the optimization problem \eqref{eq:unconstr_opt_contr} subject to the point-wise inequality
constraints
\[
   c_j(t, x(t), u(t))\leq 0, \quad  j\in \{1,\ldots,k\}, \quad \text{for almost all $t\in[0,T]$},
\]
where $c_j(t,x,u)$ is $\CC^2$ in $(x,u)$ and $\CC^1$ in $t$.
%
The main idea proposed in \cite{JH-AS:06} is to approximate the solution of the
constrained problem by solving an unconstrained optimal control problem through
a suitable translation of the well known barrier function method used in finite
dimension convex optimization \cite{SB-LV:04}. The direct translation to
infinite dimension would be
\begin{equation}
  \begin{split}
    \minimize  &~ \int_0^T l(\tau, x(\tau), u(\tau)) -\EpsBeta \sum_j \log (-c_j(t, x(\tau), u(\tau)))\, d\tau + m(x(T))\\
    \subj      &~ \dot{x}(t) = f(x(t), u(t)), \qquad x(0)=x_0\\
  \end{split}
  \label{eq:opt_barrier_log}
\end{equation}
A key difficulty of the problem in \eqref{eq:opt_barrier_log} is that the cost
functional can not be evaluated at infeasible curves. The problem is resolved by
introducing the approximate barrier function $\beta_{\DelBeta}(\cdot)$,
$0<\DelBeta\leq 1$, defined as
\begin{equation*}
  \begin{split}
    \beta_{\DelBeta}(z) =
    \begin{cases}
      -\log z & z > \DelBeta\\[1.2ex]
      \frac{k-1}{k} \left[ \left(\frac{z - k \DelBeta}{(k-1) \DelBeta} \right)^k
        - 1 \right] - \log \DelBeta & z \leq \DelBeta.
    \end{cases}
  \end{split}
  \label{eq:beta-def}
\end{equation*}
The associated barrier functional is
\begin{equation}
  b_{\DelBeta}(\xi) = \int_0^T \sum_j
  \beta_{\DelBeta}(-c_j(\tau, \alpha(\tau), \mu(\tau))),
  \label{eq:barrier-functional}
\end{equation}
which is well defined for any curve $\xi\in \Xtilde$, so that we get the optimal
control problem relaxation
\begin{equation*}
  \begin{split}
    \min_{\xi\in\TT} h(\xi) + \EpsBeta b_{\DelBeta}(\xi).
  \end{split}
\end{equation*}

\begin{remark}[Projection operator Newton method to solve the relaxed problem]
  The projection operator Newton method can be used to optimize the functional
  $g_{\EpsBeta, \DelBeta}(\xi) = h(\PP(\xi)) + \EpsBeta b_{\DelBeta}(\PP(\xi))$
  as part of a continuation method on the parameters $\EpsBeta$ and
  $\DelBeta$. The technique is to start with a large $\EpsBeta$ and $\DelBeta$,
  solve the problem $\min g_{\EpsBeta, \DelBeta}(\xi)$ using the Newton method
  starting from the current trajectory and then reduce $\EpsBeta$ and
  $\DelBeta$.
  It is worth noting that for a fixed $\EpsBeta$ it is possible to iterate on
  $\DelBeta$ up to a value for which the solution is the same as the pure
  barrier functional. For this reason in the paper we neglect the
  dependence of $g_{\EpsBeta, \DelBeta}(\xi)$ on $\DelBeta$ (thus writing
  $g_{\EpsBeta}(\xi)$).
\end{remark}


\end{document}